\documentclass[10pt]{amsart}

\usepackage{fullpage,amsmath,amsthm,amsfonts,amssymb,
graphicx,amscd,float,hyperref,enumitem,setspace,amsaddr}
\usepackage{comment}
\usepackage{xypic}
\usepackage[T1]{fontenc}
\usepackage[utf8]{inputenc}
\usepackage{mathtools}
\usepackage{aliascnt}

\hypersetup{colorlinks=true,linkcolor=blue,citecolor=blue,urlcolor=blue}

\newtheorem{thm}{Theorem}[section]

\newaliascnt{lem}{thm}
\newtheorem{lem}[lem]{Lemma}
\aliascntresetthe{lem}

\newaliascnt{qst}{thm}

\aliascntresetthe{qst}

\newaliascnt{cor}{thm}
\newtheorem{cor}[cor]{Corollary}
\aliascntresetthe{cor}

\newaliascnt{prop}{thm}
\newtheorem{prop}[prop]{Proposition}
\aliascntresetthe{prop}

\newaliascnt{ex}{thm}
\newtheorem{ex}[ex]{Example}
\aliascntresetthe{ex}

\newaliascnt{nt}{thm}
\newtheorem{nt}[nt]{Notation}
\aliascntresetthe{nt}

\newaliascnt{cl}{thm}

\aliascntresetthe{cl}

\theoremstyle{definition}
\newaliascnt{rem}{thm}
\newtheorem{rem}[rem]{Remark}
\aliascntresetthe{rem}

\newaliascnt{conv}{thm}
\newtheorem{conv}[conv]{Convention}
\aliascntresetthe{conv}

\newaliascnt{defn}{thm}
\newtheorem{defn}[defn]{Definition}
\aliascntresetthe{defn}

\newaliascnt{pd}{thm}
\newtheorem{pd}[pd]{Proposition-Definition}
\aliascntresetthe{pd}

\newaliascnt{prob}{thm}
\newtheorem{prob}[prob]{Problem}
\aliascntresetthe{prob}

\usepackage[nameinlink,capitalize]{cleveref}
\crefname{thm}{Theorem}{Theorems}
\Crefname{thm}{Theorem}{Theorems}
\crefname{prop}{Proposition}{Propositions}
\Crefname{prop}{Proposition}{Propositions}
\crefname{lem}{Lemma}{Lemmas}
\Crefname{lem}{Lemma}{Lemmas}
\crefname{cor}{Corollary}{Corollaries}
\Crefname{cor}{Corollary}{Corollaries}
\crefname{rem}{Remark}{Remarks}
\Crefname{rem}{Remark}{Remarks}
\crefname{ex}{Example}{Examples}
\Crefname{ex}{Example}{Examples}
\crefname{defn}{Definition}{Definitions}
\Crefname{defn}{Definition}{Definitions}
\crefname{prob}{Problem}{Problems}
\Crefname{prob}{Problem}{Problems}
\crefname{qst}{Question}{Questions}
\Crefname{qst}{Question}{Questions}
\crefname{cl}{Claim}{Claims}
\Crefname{cl}{Claim}{Claims}
\crefname{nt}{Notation}{Notations}
\Crefname{nt}{Notation}{Notations}
\crefname{conv}{Convention}{Conventions}
\Crefname{conv}{Convention}{Conventions}
\crefname{pd}{Proposition-Definition}{Proposition-Definitions}
\Crefname{pd}{Proposition-Definition}{Proposition-Definitions}

\newcommand{\SL}{\operatorname{SL}}
\newcommand{\Tr}{\operatorname{Tr}}
\newcommand{\C}{\mathbb C}
\newcommand{\Z}{\mathbb Z}

\newcommand{\norm}[1]{\left\lVert #1\right\rVert}
\newcommand{\cyc}{\operatorname{cyc}}
\newcommand{\supp}{\operatorname{supp}}
\renewcommand{\phi}{\varphi}

\newcommand{\bit}{\operatorname{bit}}

\makeatletter
\@namedef{subjclassname@2020}{%
  \textup{2020} Mathematics Subject Classification}
\makeatother

\title[On quantitative aspects of trace polynomials]{On quantitative aspects of trace polynomials}

\author{Ilya Kapovich}

\address{Department of Mathematics and Statistics, Hunter College of CUNY\newline
  \indent 695 Park Ave, New York, NY 10065, U.S.A.
  \newline \indent e-mail {\tt ik535@hunter.cuny.edu}
  \newline\indent  ORCID 0000-0002-7694-6236
  }

\date{}

\begin{document}

\begin{abstract}
By the classic results of Fricke and Klein~\cite{FrickeKlein}, for every word $w$ in the free group $F(a,b)$ there exists a unique integer
\emph{trace polynomial} $f_w(x,y,z)\in \Z[x,y,z]$ such that
\[
\Tr(w(A,B))=f_w(\Tr A,\Tr B,\Tr AB)
\]
for all $A,B\in \SL(2,\C)$.  In this paper we study quantitative aspects of trace polynomials.  We prove an exact formula for the leading homogeneous part of $f_w$ for every nontrivial cyclically reduced word $w\in F(a,b)$.  In particular, if $w=u_1\cdots u_n$ is cyclically reduced over $\{a,a^{-1},b,b^{-1}\}$, and if $N_{rs}(w)$ denotes the number of cyclic occurrences of the adjacent pair $rs$, then
\[
 \deg f_w=n-N_{ab}(w)-N_{b^{-1}a^{-1}}(w)
      =n-\frac{1}{2}\bigl(N_{ab}(w)+N_{ba}(w)+N_{a^{-1}b^{-1}}(w)+N_{b^{-1}a^{-1}}(w)\bigr).
\]
Consequently we obtain sharp general bounds $\lceil n/2\rceil\le \deg f_w\le n$ for arbitrary $w\in F(a,b)$ with cyclically reduced length $n$, and an exact formula $\deg f_w=n-s$ for positive words $w=a^{\alpha_1}b^{\beta_1}\dots a^{\alpha_s}b^{\beta_s}$.  We also study $\deg f_w$ for random positive words and for random freely reduced and random cyclically reduced words.  We also obtain explicit exponential upper bounds for the growth of the $\ell_1$ and $\ell_\infty$ norms of $f_w$ and exhibit examples with exponential coefficient growth at rate $\varphi^n$, where $\varphi$ is the golden ratio.  We show that for random freely reduced, random cyclically reduced and random positive words $w_n$ of length $n$ in $F(a,b)$ the size of the support of $f_{w_n}$ grows at least quadratically in $n$ and the total bit-size of the nonzero coefficients of $f_{w_n}$ grows at least as constant multiple of $n^3$. Consequently, any algorithm computing $f_w$ in totally expanded form has worst-case time complexity as well as generic-case time complexity for the above models bounded below by $\Omega(n^3)$.  We also give a deterministic algorithm which computes the fully expanded polynomial $f_w$ in time $O(n^5)$ and space $O(n^4)$, in terms of the input word length $n$.  As a consequence, $\SL(2,\C)$-character equivalence for elements $v,w\in F(a,b)$ is decidable in time $O(n^5)$, where $n$ is the maximum of the lengths of $v$ and $w$.
\end{abstract}

\maketitle

\tableofcontents

\section{Introduction}\label{sec:introduction}

Trace polynomials for group words in two matrices were introduced in the classic 1897 work of Fricke and Klein~\cite{FrickeKlein} and they play a fundamental role in the study of characters of two-generator groups.  Fricke and Klein showed that for every element $w\in F(a,b)$ there exists a polynomial $f_w(x,y,z)\in \Z[x,y,z]$, called the \emph{trace polynomial of $w$} such that
\[
\Tr(w(A,B))=f_w(\Tr A,\Tr B,\Tr AB)
\]
for all $A,B\in \SL(2,\C)$. Moreover, $f_w$ is known to be uniquely determined by $w$, see \cite{Horowitz1972}. The three trace functions
\[
  x=\Tr A,\qquad y=\Tr B,\qquad z=\Tr(AB)
\]
provide affine polynomial coordinates on the affine GIT quotient
\[
 \operatorname{Hom}(F(a,b),\SL(2,\C))//\SL(2,\C),
\]
namely the $\SL(2,\C)$-character variety of $F(a,b)$. This quotient is isomorphic to
$\C^3$, and its coordinate ring is naturally identified with the polynomial ring
$\C[x,y,z]$. Thus the single polynomial $f_w$ is the coordinate expression of
the regular trace function associated to $w$ in these Fricke coordinates. This
special three-coordinate description is one of the main reasons why the rank-two
$\SL(2,\C)$ setting permits the explicit degree, coefficient-growth, and algorithmic
estimates studied below. Since in $\SL(2,\C)$ traces of matrices are preserved by inversion and conjugation, the trace polynomial $f_w$ for $w\in F(a,b)$ is also invariant under conjugating and inverting $w$ in $F(a,b)$. 


Trace polynomials naturally appear in the study of character varieties and trace algebras,
including Horowitz's work on characters of free groups \cite{Horowitz1972}, the
Culler--Shalen theory of character varieties \cite{CullerShalen1983}, Goldman's work
on trace functions and surface-group character varieties \cite{Goldman1986}, and the
general invariant-theoretic viewpoint of Procesi \cite{Procesi1976} and Sikora
\cite{Sikora2012}.   The algorithmic constructive computation of
$f_w$ from trace identities is also standard, but the precise computational complexity of computing $f_w$ does not seem to have been addressed in the literature so far.

Apart from Horowitz's work, trace-polynomial computations also appear explicitly
in algorithmic invariant-theory contexts, for example in the $L_2$-quotient
algorithm of Plesken--Fabia\'nska and its generalization by Jambor
\cite{PleskenFabianska2009,Jambor2015}.  However, in those works the emphasis is on
constructing and using trace-coordinate descriptions of representation varieties,
rather than on the computational complexity of the fully expanded polynomial for a given input word $w\in F(a,b)$.

There are two closely related bodies of literature which provide additional motivation
for the questions considered here.  First, trace equivalence of individual words is
closely related to the problem of distinguishing conjugacy classes by matrix traces.  Horowitz
constructed non-conjugate words with the same $\SL(2,\C)$ trace polynomial
\cite{Horowitz1972}.  In the terminology used later in this paper, two words
$v,w\in F(a,b)$ are $\SL(2,\C)$-character equivalent precisely when
$f_v=f_w$.  Trace-polynomial identities also occur naturally in the study of
translation equivalence in free groups: Kapovich--Levitt--Schupp--Shpilrain
identify trace identities as one of the sources of translation-equivalent elements
and relate translation equivalence to Whitehead-graph data
\cite{KapovichLevittSchuppShpilrain2007}.  Lawton, Louder, and McReynolds studied
higher-rank trace equivalence in connection with decision problems, complexity,
and representation theory \cite{LawtonLouderMcReynolds2017}; in particular, their
work emphasizes that the classical $\SL(2,\C)$ phenomena do not automatically persist
in rank $3$ and higher.  Character equivalence is equality of trace functions, or equivalently equality of trace polynomials in the rank-two $\SL(2,\C)$ case, whereas translation equivalence concerns equality of translation lengths in all free isometric actions on $\mathbb R$-trees.  These notions are related through trace identities and Whitehead-graph data, but they are not identical.
Second, for $\SL(3,\C)$ the character varieties of free groups and the corresponding
trace-coordinate rings are much more complicated than the three-coordinate
$\SL(2,\C)$ case.  Lawton gave explicit generators and relations for pairs of
$3\times3$ unimodular matrices \cite{Lawton2007Pairs}, determined minimal affine
coordinates for $\SL(3,\C)$ free-group character varieties \cite{Lawton2008Minimal},
and studied algebraic independence of natural $\SL(3,\C)$ trace coordinates
\cite{Lawton2010AlgInd}.  These higher-rank results show why the elementary
three-variable Fricke setting treated here is unusually tractable, and they provide
natural targets for possible analogues of the degree, coefficient-growth, and
complexity estimates proved below.

In this paper we study several quantitative aspects of $f_w$ in terms of the input word $w$.  These include the degree of $f_w$, the growth of its $\ell_1$ and $\ell_\infty$ coefficient norms, the size of its support, and the bit-size of its coefficients, both in worst-case examples and in natural random models.  We also address the deterministic complexity of computing $f_w$ in fully expanded form.  The output-size lower bounds below come from explicit coefficient slices and therefore apply to any algorithm which actually writes the expanded sparse polynomial.


We now state the main results proved below.  We fix the free basis $\mathcal{X}=\{a,b\}$ for the free group of rank two, $F_2=F(\mathcal{X})=F(a,b)$. 

For a word $w$ over $\mathcal{X}^{\pm 1}$ we denote by $|w|$ the length of $w$, we denote by $|w|_{\mathcal X}$ the freely reduced length of $w$, and we denote by $||w||_{\mathcal X}$ the cyclically reduced length of $w$.  For a nontrivial cyclically reduced word $w\in F(a,b)$ the \emph{syllable length} $||w||_{syl}$ of $w$ is the number of maximal $a$-power and $b$-power subwords in $w$, counted cyclically. E.g. $||a^3||_{syl}=1$, $||a^9b^{-1}||_{syl}=2$ and $||ab^3a^4||_{syl}=2$. Note that if $w$ is not an $a$-power and not a $b$-power then $||w||_{syl}\ge 2$ is even.

We first obtain an exact leading-term formula for arbitrary cyclically reduced words.  Let the letters of a cyclically reduced word be read cyclically.  For letters
$r,s\in\{a,a^{-1},b,b^{-1}\}$ let $N_{rs}(w)$ be the number of cyclic indices $i$ for which
$u_i=r$ and $u_{i+1}=s$.  The proof of the exact leading-term statement in
\cref{thm:intro-degree-bounds}(1) is given in \cref{prop:top-homogeneous-exact-degree}; the consequent sharp
bounds in \cref{thm:intro-degree-bounds}(2) are proved in \cref{thm:degree-lower}.

\begin{thm}\label{thm:intro-degree-bounds}
Let $w=u_1\cdots u_n\in F(a,b)$ be a nontrivial cyclically reduced word over
$\{a,a^{-1},b,b^{-1}\}$.  Put
\[
 m_a=\#\{i:u_i\in\{a,a^{-1}\}\},\qquad
 m_b=\#\{i:u_i\in\{b,b^{-1}\}\},\qquad
 R=N_{ab}(w)+N_{b^{-1}a^{-1}}(w).
\]
Then the following hold.
\begin{enumerate}
\item The top homogeneous part of $f_w$ is
\[
 [f_w]_{\rm top}=(-1)^{N_{ab}(w)+N_{ba}(w)}
 x^{m_a-R}y^{m_b-R}z^R.
\]
In particular,
\[
 \deg f_w=n-R
 =n-\frac{1}{2}\bigl(N_{ab}(w)+N_{ba}(w)+N_{a^{-1}b^{-1}}(w)+N_{b^{-1}a^{-1}}(w)\bigr).
\]
\item Consequently,
\[
\left\lceil \frac{n}{2}\right\rceil\le \deg f_w\le n,
\]
and, for every $n\ge 1$, both the lower and upper bounds are attained by some cyclically reduced word of length $n$.
\end{enumerate}
\end{thm}

In the probabilistic statements below, the notation $O_D(\cdot)$ means that the
implicit multiplicative constant may depend on the fixed parameter $D>0$, but not on $n$.


For positive words in $F(a,b)$ we obtain both a precise result and a probabilistic result about $\deg f_w$, see \cref{prop:positive-degree,cor:random-positive-degree}:

\begin{thm}\label{thm:intro-positive-degree}
The following hold:
\begin{enumerate}
\item Let
\[
 w=a^{\alpha_1}b^{\beta_1}\cdots a^{\alpha_s}b^{\beta_s}
\]
be a positive word in $F(a,b)$ with $s\ge 1$ and all $\alpha_i,\beta_i>0$. Let $n=||w||_{\mathcal X}=\sum_{i=1}^s (\alpha_i+\beta_i)$.

Then
\[
 \deg f_w=n-s.
\]
\item Let $0<p<1$.  Let $w_n$ be a random positive word in $F(a,b)$ of length $n$ with independently chosen 
letters satisfying $\Pr(a)=p$ and $\Pr(b)=1-p$. Then for every $D>0$,
\[
 \deg f_{w_n}=\bigl(1-p(1-p)\bigr)n+O_D(\sqrt{n\log n})
\]
with probability at least $1-O_D(n^{-D})$ as $n\to\infty$. 

The degenerate cases $p=0$ and $p=1$ are one-generator cases and have degree exactly $n$.
\end{enumerate}
\end{thm}

For random freely reduced and random cyclically reduced words in $F(a,b)$ we prove an exact typical-degree asymptotic, \cref{prop:random-freely-reduced-degree,cor:random-cyclically-reduced-degree}:

\begin{thm}\label{thm:intro-random-freely-cyclically-reduced}
Let $W_n\in F(a,b)$ be a random freely reduced word of length $n$ in $F(a,b)$, generated
by the standard nonbacktracking Markov chain.  Then for every $D>0$,
\[
 \deg f_{W_n}= \frac{5n}{6}+O_D(\sqrt{n\log n})
\]
with probability at least $1-O_D(n^{-D})$.  The same estimate holds for uniformly
random cyclically reduced words of length $n$.
\end{thm}

The following statement, regarding the possible values of the constant term of $f_w$, is probably well known informally, but we prove it for completeness, see \cref{thm:constant-term}:

\begin{thm}\label{thm:intro-constant-term}
For every $w\in F(a,b)$,
\[
 f_w(0,0,0)\in\{-2,0,2\}.
\]
\end{thm}

\phantomsection\label{not:coefficient-bit-norms}%
For a polynomial $f\in \Z[x,y,z]$ and integers $i,j,k\ge 0$, denote by $[x^iy^jz^k]f$ the coefficient of the monomial $x^iy^jz^k$ in $f$.
Also, for $f\in \Z[x,y,z]$ denote
\[
 \norm{f}_\infty=\max_{i,j,k}|[x^iy^jz^k]f|
\]
and
\[
 \norm{f}_1=\sum_{i,j,k}|[x^iy^jz^k]f|.
\]
Also, for a nonzero integer $m\ne 0$, denote by $\bit(m)$ the binary length, or bit-length, of $m$ (disregarding the possible minus sign, so that $\bit(m)=\bit(-m)$).  We set $\bit(0)=0$ only when this convention is explicitly needed.  For $f\ne 0$ in $\Z[x,y,z]$ denote
\[
 \norm{f}_{\bit,\infty}=\max_{[x^iy^jz^k]f\ne 0}\bit([x^iy^jz^k]f),
\]
with the convention that $\norm{0}_{\bit,\infty}=0$, and define
\[
 \norm{f}_{\bit,1}=\sum_{[x^iy^jz^k]f\ne 0} \bit([x^iy^jz^k]f).
\]
Thus both bit-size quantities are taken only over nonzero coefficients, except for the explicit zero-polynomial convention above.

Put $\varphi=\frac{1+\sqrt5}{2}$, the golden ratio, and put
\[
 \varphi_0=\sqrt{2\varphi^2+1}.
\]
Note that 
\[
1.6<\varphi<\varphi_0<2.5
\]

We obtain an exponential upper bound for the growth of coefficients of $f_w$ and show that such exponential growth does in fact occur, see \cref{thm:coefficient-upper,thm:coefficient-lower}:

\begin{thm}\label{thm:intro-coeff-growth}
The following hold:
\begin{enumerate}
\item For every nontrivial $w\in F(a,b)$ with $n=||w||_{\mathcal X}$ we have
\[
 \norm{f_w}_\infty\le 2\varphi_0^n,
 \qquad
 \norm{f_w}_1\le 2\binom{n+3}{3}\varphi_0^n.
\]
\item There exists $c>0$ such that for words $w=a^n$, where $n\ge 1$, we have
\[
\norm{f_{a^n}}_1\ge c \varphi^n
\]
and 
\[
\norm{f_{a^n}}_\infty\ge c\varphi^n/\sqrt n.
\]
\end{enumerate}
\end{thm}

In \cref{thm:fricke-existence,thm:recursive-algorithm} below we recall the standard inductive algorithm for computing $f_w$ in terms of $w$, with integer coefficients recorded in bit-form. In addition, we provide a different algorithm for computing $f_w$, via multiplying matrices with indeterminate coefficients, and use it to obtain a polynomial time bound for computing $f_w$, see \cref{thm:polytime-expanded} below:

\begin{thm}[Polynomial-time expanded computation]\label{thm:intro-polynomial-time}
There is an algorithm which, given a cyclically reduced word $w$ of length $n=||w||_{\mathcal X}$ in $F(a,b)$, outputs the fully expanded polynomial
\[
 f_w(x,y,z)\in\Z[x,y,z]
\]
in deterministic Turing-machine time $O(n^5)$ and space $O(n^4)$. 
In this dense coefficient-table representation, the output itself has $O(n^4)$ bits.
\end{thm}

We next state the support and sparse-output lower bounds.  The coefficient bit-size notation $\norm{\cdot}_{\bit,1}$ is defined in the \hyperref[not:coefficient-bit-norms]{notation paragraph above} and is used throughout the later output-size estimates; it denotes the total binary bit-size of the nonzero coefficients in the expanded sparse form.  For a complex polynomial $f$ we denote by $\supp(f)$ the set of monomials that occur with nonzero coefficients in $f$.  For a variable $t$ and an integer $m\ge 1$ denote by $E_m(t)=U_{m-1}(t/2)$, where $U_j$ is the Chebyshev polynomial of the second kind.  As observed in Section~\ref{sec:support-lower}, $E_m(t)$ is an integer polynomial of degree $m-1$ with $\#\supp(E_m(t))= \left\lceil\frac{m}{2}\right\rceil$.
The following deterministic family is proved using a single explicit coefficient slice; it gives both quadratic support growth and a cubic sparse-output lower bound.

\begin{thm}\label{thm:intro-support-lower}
Consider the words 
\[
 w_m=a^m b^m (ab)^m,
\]
in $F(a,b)$ for all $m\ge 1$, so that $|w_m|_{\mathcal X}=||w_m||_{\mathcal X}=4m$. Then:
\begin{enumerate}
\item The $\Z[x,y]$-coefficient $g_m(x,y)$ of $z^{m+1}$ in $f_{w_m}$ is equal to
\[
 E_m(x)E_m(y),
\]
\item We have
\[
 \#\supp(f_{w_m})=\Omega(m^2).
\]
\item The binary bit-size $\norm{g_m}_{\bit,1}$ of $g_m(x,y)$ is $\Omega(m^3)$.  In particular,
\[
 \norm{f_{w_m}}_{\bit,1}=\Omega(m^3).
\]
\item Any algorithm which outputs the trace polynomial $f_w$ in fully expanded sparse binary form has worst-case output size, and therefore worst-case running time, at least $\Omega(n^3)$ for words of length $n$.
\end{enumerate}
\end{thm}

The deterministic example is complemented by the following exponentially generic lower bounds for random positive, random freely reduced, and uniformly random cyclically reduced words, proved in \cref{thm:random-positive-support,thm:random-freely-reduced-support,cor:random-expanded-output-lower}.

\begin{thm}[Quadratic random support, cubic-size random coefficients, and cubic sparse-output lower bounds]\label{thm:intro-random-support-output}
Let $W_n$ be a random word of length $n$, either in the positive model on $\{a,b\}$ with $\Pr(a)=p\in(0,1)$, or in the standard nonbacktracking freely reduced model on $\{a,a^{-1},b,b^{-1}\}$.  Then the following hold.
\begin{enumerate}
\item There are constants $c,C,\lambda>0$, depending on the model and on $p$ in the positive case, such that
\[
 \Pr\bigl(\#\supp(f_{W_n})\ge c n^2\bigr)\ge 1-Ce^{-\lambda n}.
\]
\item There are constants $c_1,c_2,C,\lambda>0$, depending on the model and on $p$ in the positive case, such that
\[
 \Pr\left(\exists i,j,k\text{ with } c_1n^3\le
 |[x^iy^jz^k]f_{W_n}|\le c_2n^3\right)\ge 1-Ce^{-\lambda n}.
\]
\item There are constants $c_3,C_3,\lambda_3>0$, depending on the model and on $p$ in the positive case, such that
\[
 \Pr\bigl(\norm{f_{W_n}}_{\bit,1}\ge c_3 n^3\bigr)\ge 1-C_3e^{-\lambda_3 n}.
\]
Thus the total binary bit-size of the nonzero coefficients in the fully expanded sparse output grows at least cubically with exponentially high probability.
\item Fix one of these two random models.  There are constants $c_4,C_4,\lambda_4>0$, depending on the model, on $p$ in the positive case, and on the precise sparse-output convention, such that for every deterministic algorithm $\mathcal A$ which computes and outputs $f_w$ in fully expanded sparse binary form on input words $w$ of length $n$ in the same model, if $T_{\mathcal A}(W_n)$ denotes the running time of $\mathcal A$ on input $W_n$, then
\[
 \Pr\bigl(T_{\mathcal A}(W_n)\ge c_4 n^3\bigr)\ge 1-C_4e^{-\lambda_4 n}
\]
for all $n\ge1$.
\item There are constants $c_5,C_5,\lambda_5>0$, depending on the precise sparse-output convention, such that for every deterministic algorithm $\mathcal A$ which computes and outputs $f_w$ in fully expanded sparse binary form on input cyclically reduced words $w$ of length $n$, if $U_n$ is uniformly random among cyclically reduced words of length $n$ and $T_{\mathcal A}(U_n)$ denotes the running time of $\mathcal A$ on input $U_n$, then
\[
 \Pr\bigl(T_{\mathcal A}(U_n)\ge c_5 n^3\bigr)\ge 1-C_5e^{-\lambda_5 n}
\]
for all $n\ge1$.
\end{enumerate}
In particular, every such algorithm has an $\Omega(n^3)$ running-time lower bound with exponentially high probability in the positive, freely reduced, and uniformly cyclically reduced models.
\end{thm}

Although \cref{thm:intro-random-support-output} formally implies the existence of words with $\norm{f_w}_{\bit,1}=\Omega(n^3)$, \cref{thm:intro-support-lower} is included for a distinct reason: it gives an explicit deterministic sequence of words with cubic total coefficient bit-size growth.  Thus \cref{thm:intro-support-lower} supplies a concrete worst-case construction, while \cref{thm:intro-random-support-output} gives a probabilistic, exponentially generic lower bound. Informally, in the terminology of \cite{KMSS2003}, these results say that the exponentially generic-case running-time lower bound for the fully expanded sparse binary output problem of computing $f_w$ is of cubic order for random positive, random freely reduced, and uniformly random cyclically reduced words in $F(a,b)$.

Recall that elements $v,w\in F(a,b)$ are called \emph{$\SL(2,\C)$-character equivalent} or just \emph{character equivalent}, denoted $v\equiv_c w$, if for every $A,B\in \SL(2,\C)$ we have $\Tr v(A,B)=\Tr w(A,B)$. It is easy to see that for $v,w\in F(a,b)$ we have $v\equiv_c w$ if and only if $f_v=f_w$. Character equivalence is closely related to the study of geometry of curves on surfaces, e.g. see \cite{Lein2003}, as well as to translation equivalence in free groups~\cite{KapovichLevittSchuppShpilrain2007}.

As a direct algorithmic consequence of \cref{thm:intro-polynomial-time} we also obtain the following
decision result, see \cref{cor:polytime-character-equivalence}:

\begin{cor}\label{cor:intro-character-equivalence-decision}
There is a deterministic algorithm which, given two freely reduced words
$v,w\in F(a,b)$ of length at most $n$, decides whether or not
$v\equiv_c w$, that is, whether or not $v$ and $w$ are
$\SL(2,\C)$-character equivalent, in time $O(n^5)$ under the dense-output
Turing-machine model used in \cref{thm:intro-polynomial-time}.
\end{cor}

We include a list of open problems in Section~\ref{sec:open-problems}.
\bigskip

{\bf Acknowledgements.} The author is grateful to Alex Wright for suggesting improvements to the results of the original version of this paper.

\bigskip

\section{The trace polynomial and the basic trace identities}

Let $F_2=F(a,b)$ be the free group on two generators.  The existence and basic properties of trace polynomials rely on the following classic trace identities for $\SL(2,\C)$, see, e.g. \cite{Horowitz1972}:

\begin{prop}\label{prop:basic}
The following hold:
\begin{enumerate}
\item[(a)] For any $U,V\in \SL(2,\C)$ we have 
\begin{equation}\label{eq:basic-trace-identity}
 \Tr(UV)+\Tr(UV^{-1})=\Tr(U)\Tr(V)
\end{equation}
\item[(b)] 
For any $U\in \SL(2,\C)$ we have 
\begin{equation}\label{eq:basic-trace-identity2}
 \Tr(U)=\Tr(U^{-1})
\end{equation}
\item[(c)]
For any $U,V\in \SL(2,\C)$ we have 
\begin{equation}\label{eq:basic-trace-identity3}
 \Tr(VUV^{-1})=\Tr(U)
\end{equation}
\end{enumerate}
\end{prop}

Horowitz~\cite{Horowitz1972} used the above facts to formally introduce trace polynomials.  The following standard formulation goes back to Fricke--Klein and Horowitz~\cite{FrickeKlein,Horowitz1972}.

\begin{pd}[Fricke trace polynomial]\label{thm:fricke-existence}
For every word $w\in F(a,b)$ there exists a unique polynomial
\[
 f_w(x,y,z)\in \Z[x,y,z]
\]
such that
\[
 \Tr(w(A,B))=f_w(\Tr A,\Tr B,\Tr(AB))
\]
for all $A,B\in \SL(2,\C)$.
\end{pd}

We will revisit \cref{thm:fricke-existence} below when discussing an algorithm for computing $f_w$.

We also need to record some basic properties and examples of trace polynomials.

\begin{prop}\label{prop:basicP}
The following hold:

\begin{enumerate}
\item For any $u,w\in F(a,b)$ we have
\[
f_w=f_{uwu^{-1}}=f_{w^{-1}}
\]
\item We have
\[
f_a=x, \quad f_b=y, \text{  and  } f_{ab}=z.
\]
\item We have
\[
 f_{a^2}=x^2-2,
 \qquad
 f_{b^2}(x,y,z)=y^2-2, \quad \text{ and } f_{ab^{-1}}(x,y,z)=xy-z
\]

\item For $w=[a,b]=aba^{-1}b^{-1}$ we have
\[
 f_{[a,b]}=x^2+y^2+z^2-xyz-2.
\]
\item For $w=[a,b^{-1}]=ab^{-1}a^{-1}b$ we have
\[
 f_{[a,b^{-1}]}=x^2+y^2+z^2-xyz-2.
\]
\end{enumerate}
\end{prop}
\begin{proof}
Parts (1) and (2) follow directly from Proposition~\ref{prop:basic}.

For part (3), by the Cayley--Hamilton Theorem for every $M\in\SL(2,\C)$ we have
\[
 M^2-(\Tr M)M+I=0.
\]
Taking traces gives
\[
 \Tr(M^2)=(\Tr M)^2-2.
\]
Applying this to $M=A$ and to $M=B$ gives
\[
 f_{a^2}=x^2-2,
 \qquad
 f_{b^2}=y^2-2.
\]

Next, the basic trace identity
\[
 \Tr(UV)+\Tr(UV^{-1})=\Tr(U)\Tr(V)
\]
with $U=A$ and $V=B$ gives
\[
 z+\Tr(AB^{-1})=xy,
\]
and hence
\[
 f_{ab^{-1}}=xy-z.
\]
Thus (3) holds. 

For part (4), put
\[
 C=AB.
\]
Then
\[
 [A,B]=ABA^{-1}B^{-1}=C A^{-1}B^{-1}.
\]
Denote $\Tr(A)=x$, $\Tr(B)=y$ and $\Tr(AB)=z$. The basic trace identity \eqref{eq:basic-trace-identity} with $U=C$ and $V=AB$ gives
\[
 \Tr(CAB)+\Tr(C(AB)^{-1})=\Tr(C)\Tr(AB)=z^2.
\]
Since $C(AB)^{-1}=I$, this becomes
\begin{equation}\label{eq:comm-step-1}
 \Tr(ABAB)+2=z^2.
\end{equation}
Thus
\[
 \Tr(ABAB)=z^2-2.
\]
Now apply the basic trace identity to $U=AB$ and $V=A^{-1}B^{-1}$.  Since
\[
 UV=ABA^{-1}B^{-1}=[A,B],
\]
and
\[
 UV^{-1}=ABBA,
\]
we get
\begin{equation}\label{eq:comm-step-2}
 \Tr([A,B])+\Tr(ABBA)=\Tr(AB)\Tr(A^{-1}B^{-1})=z^2,
\end{equation}
because $\Tr(A^{-1}B^{-1})=\Tr((BA)^{-1})=\Tr(BA)=z$.
It remains to express $\Tr(ABBA)$.  By cyclic invariance of trace,
\[
 \Tr(ABBA)=\Tr(A^2B^2).
\]
Using the Cayley--Hamilton identities
\[
 A^2=xA-I,
 \qquad
 B^2=yB-I,
\]
we obtain
\[
 A^2B^2=(xA-I)(yB-I)=xyAB-xA-yB+I.
\]
Taking traces gives
\begin{equation}\label{eq:comm-step-3}
 \Tr(A^2B^2)=xyz-x^2-y^2+2.
\end{equation}
Substituting \eqref{eq:comm-step-3} into \eqref{eq:comm-step-2} yields
\[
 \Tr([A,B])=z^2-(xyz-x^2-y^2+2)=x^2+y^2+z^2-xyz-2.
\]
Therefore
\[
 f_{[a,b]}(x,y,z)=x^2+y^2+z^2-xyz-2.
\]
For part (5), apply part (4) to the pair $A,B^{-1}$.  The corresponding
Fricke coordinates are
\[
 \Tr A=x,\qquad \Tr(B^{-1})=y,\qquad \Tr(AB^{-1})=xy-z.
\]
Therefore
\[
 \Tr([A,B^{-1}])
 =x^2+y^2+(xy-z)^2-xy(xy-z)-2
 =x^2+y^2+z^2-xyz-2.
\]
This proves (5).
\end{proof}

Let $T_n(t)\in \Z[t]$ denote the Chebyshev polynomial of the first kind, normalized by
\[
 T_n(\cos\theta)=\cos(n\theta).
\]

\begin{lem}[Powers of one generator and Chebyshev polynomials]\label{lem:chebyshev-power}$ $

For every integer $n\ge 0$, we have
\[
 f_{a^n}(x,y,z)=2T_n(x/2),
 \qquad
 f_{b^n}(x,y,z)=2T_n(y/2).
\]
\end{lem}

\begin{proof}
We prove the statement for $a^n$; the proof for $b^n$ is identical.  Let
$A\in \SL(2,\C)$ and put $x=\Tr A$.  The Cayley--Hamilton identity gives
\[
 A^2-xA+I=0.
\]
Multiplying by $A^{n-1}$ and taking traces yields, for $n\ge 1$,
\[
 \Tr(A^{n+1})=x\Tr(A^n)-\Tr(A^{n-1}).
\]
Thus the polynomials
\[
 P_n(x):=f_{a^n}(x,y,z)=\Tr(A^n)
\]
satisfy
\[
 P_0(x)=2,
 \qquad
 P_1(x)=x,
 \qquad
 P_{n+1}(x)=xP_n(x)-P_{n-1}(x).
\]
On the other hand, the polynomials
\[
 Q_n(x):=2T_n(x/2)
\]
satisfy the same initial conditions and the same recurrence, because the Chebyshev
polynomials satisfy
\[
 T_{n+1}(t)=2tT_n(t)-T_{n-1}(t).
\]
Therefore $P_n(x)=Q_n(x)=2T_n(x/2)$ for all $n\ge 0$.
\end{proof}

\section{A recursive algorithm}\label{sec:algorithm}

\begin{conv}[Normalization]\label{def:normalization}
For a freely reduced word $w\in F(a,b)$,  denote by  $\widehat w$ any cyclic permutation of the cyclically reduced form of $w$ or of $w^{-1}$.
In view of Proposition~\ref{prop:basicP}, we always have 
\[
 f_w=f_{\widehat w},\qquad f_{w^{-1}}=f_w.
\]
\end{conv}

\begin{thm}[Recursive computation of $f_w$]\label{thm:recursive-algorithm}
There is a deterministic recursive algorithm which computes $f_w\in\Z[x,y,z]$ for every freely reduced word $w\in F(a,b)$.
\end{thm}

\begin{proof}
We use induction on $||w||_{\mathcal X}$.

Using Proposition~\ref{prop:basicP}, the base cases $||w||_{\mathcal X}\le 2$ are:
\[
 f_1=2,
\]
\[
 f_a=f_{a^{-1}}=x,
 \qquad
 f_b=f_{b^{-1}}=y,
\]
\[
 f_{a^2}=f_{a^{-2}}=x^2-2,
 \qquad
 f_{b^2}=f_{b^{-2}}=y^2-2,
\]
\[
 f_{ab}=f_{ba}=f_{a^{-1}b^{-1}}=f_{b^{-1}a^{-1}}=z,
\]
\[
 f_{ab^{-1}}=f_{b^{-1}a}=f_{a^{-1}b}=f_{ba^{-1}}=xy-z.
\]
These handle all freely reduced words of length at most two.

Assume now that $w\in F(a,b)$ is a freely reduced word with $||w||_{\mathcal X}=n\ge 3$ and that we already computed the trace polynomials for all words with cyclically reduced length $\le n-1$. Since the trace polynomial is invariant under conjugation, we may assume that $w$ is cyclically reduced with $|w|_{\mathcal X}=||w||_{\mathcal X}=n\ge 3$. 

Suppose that some letter
$s\in\{a,a^{-1},b,b^{-1}\}$ occurs at least twice in the (cyclically reduced) word $w$.  Note that such $s$ always exists if $n\ge 5$. Cyclically
rotate so that
\[
 w=sPsQ,
\]
where $P,Q$ are possibly empty reduced words.  Put
\[
 U=sP,
 \qquad
 V=sQ.
\]
Then $UV=w$.  By \eqref{eq:basic-trace-identity},
\[
 f_w=f_U f_V-f_{UV^{-1}}.
\]
But
\[
 UV^{-1}=sP(sQ)^{-1}=sPQ^{-1}s^{-1}
\]
is conjugate to $PQ^{-1}$.  Hence
\begin{equation}\label{eq:main-recursion}
 f_w=f_{sP}f_{sQ}-f_{PQ^{-1}}.
\end{equation}
Both $sP$ and $sQ$ have length strictly smaller than $n$, and $PQ^{-1}$ has length at most
$n-2$ before free reduction.  Hence, by the inductive hypothesis, $f_{sP}, f_{sQ}, f_{PQ^{-1}}$ are already known. Then \eqref{eq:main-recursion} computes $f_w$.

It remains to discuss the case where no letter occurs twice in the cyclic word.  Since
there are only four letters $a,a^{-1},b,b^{-1}$, the length is at most four.  Lengths
zero, one and two have already been treated.  A cyclically reduced word of length
three with no repeated letter cannot occur: among three distinct letters from
$\{a,a^{-1},b,b^{-1}\}$, either two inverse letters are adjacent, giving free reduction,
or they occur at the cyclic ends, violating cyclic reduction.  Therefore the only exceptional
case is length four, where all four letters occur exactly once.  Then $w$ is conjugate to one of $[a,b]^{\pm 1}$ or $[a,b^{-1}]^{\pm 1}$.  Indeed, after a cyclic rotation and replacing $w$ by $w^{-1}$ if necessary, we may assume that the first letter is $a$; cyclic reduction then leaves only the cyclic orders
\[
 a b a^{-1} b^{-1} \quad\hbox{and}\quad a b^{-1} a^{-1} b,
\]
which represent $[a,b]$ and $[a,b^{-1}]$, respectively.  Hence
\[
f_w=x^2+y^2+z^2-xyz-2. 
\]
This completes the proof.
\end{proof}

In Section~\ref{sec:polytime} below we will provide a different, matrix-based, algorithm for computing $f_w$. However, we use the recursive algorithm described in Theorem~\ref{thm:recursive-algorithm} to obtain results about quantitative properties of trace polynomials.

\section{Degree bounds}

\begin{conv}For a nontrivial cyclically reduced word $w$ in $F(a,b)$, which is not a power of $a$ or $b$, define the \emph{syllable length} $||w||_{syl}$ of $w$ as the number of transitions, counted cyclically, between nonzero powers of $a$ and $b$. For $w=a^n$ or $w=b^n$, where $n\ne 0$, we define the syllable length of $w$ to be $1$.

Any nontrivial cyclically reduced word $w$ in $F(a,b)$, which is not a power of $a$ or $b$, can be written, up to a cyclic permutation, in \emph{standard syllable form}
\[
 w=a^{\alpha_1}b^{\beta_1}\cdots a^{\alpha_s}b^{\beta_s},
\]
where $s\ge 1$ and $\alpha_i,\beta_i\ne 0$.

Note that in this case the syllable length of $w$ is $2s$, and we also have
 \[
 n=||w||_{\mathcal X}=\sum_i |\alpha_i| +\sum_i |\beta_i|.
 \]
 Moreover, $s\le n/2$ and the equality occurs if and only if all $\alpha_i,\beta_i=\pm 1$.
\end{conv}

For a nonzero polynomial $p\in\Z[x,y,z]$, let $\deg p$ denote total degree.

\begin{prop}[Upper bound]\label{prop:degree-upper}
For every $w\in F(a,b)$, we have
\[
 \deg f_w\le ||w||_{\mathcal X}
\]
\end{prop}

\begin{proof}
Since trace polynomials are invariant under conjugation, we may first replace $w$ by a cyclically reduced conjugate.  Thus $|w|=||w||_{\mathcal X}$ for the word to which the recursive argument is applied.  The proof is then by induction on this length, using the recursion of \cref{thm:recursive-algorithm}.  The base cases are immediate.  In the main recursive step
\[
 f_w=f_{sP}f_{sQ}-f_{PQ^{-1}},
\]
we have
\[
 |sP|+|sQ|=|w|,
 \qquad
 |PQ^{-1}|\le |w|-2.
\]
By induction,
\[
 \deg(f_{sP}f_{sQ})\le |sP|+|sQ|=|w|,
\]
and
\[
 \deg f_{PQ^{-1}}\le |w|-2.
\]
The exceptional length-four step is finite and follows from the same identity; equivalently, it is the finite case where the recursion uses the commutator polynomial formula from the base cases.  Thus
$\deg f_w\le ||w||_{\mathcal X}$ for every $w\in F(a,b)$.
\end{proof}

\begin{rem}\label{rem:sharpDB}
The upper bound obtained in Proposition~\ref{prop:degree-upper} is sharp.

Indeed, by Lemma~\ref{lem:chebyshev-power}, for every $n\ge 1$ 
\[
f_{a^n}=2T_n(x/2)
\]
where $T_n$ is the Chebyshev polynomial of the first kind, which has degree $n$. Thus
\[
\deg f_{a^n}=n=||a^n||_{\mathcal X}.
\]
\end{rem}

We now recall a version of Horowitz's "generic triangular specialization" method from \cite{Horowitz1972}.  The strengthened degree formula below does not require this specialization, but the lemma gives useful context for the earlier degree estimates and for the relationship between trace-polynomial degree and syllable structure.

\begin{lem}[The Horowitz triangular specialization]\label{lem:horowitz-specialization-precise}
Let
\[
 w=a^{\alpha_1}b^{\beta_1}\cdots a^{\alpha_s}b^{\beta_s}
\]
be cyclically reduced, with all $\alpha_i,\beta_i\ne0$.  Here \(\lambda\) and \(\mu\) are algebraically independent indeterminates.  Put
\[
 K=\mathbb Q(\lambda,\mu),
 \qquad
 R=K[t].
\]
In $\SL(2,R)$ define
\[
 A=\begin{pmatrix}\lambda&t\\0&\lambda^{-1}\end{pmatrix},
 \qquad
 B=\begin{pmatrix}\mu&0\\t&\mu^{-1}\end{pmatrix}.
\]
Then $\Tr(w(A,B))\in K[t]$ has $t$-degree exactly $2s$, and its $t^{2s}$ coefficient is
\begin{equation}\label{eq:horowitz-leading-precise}
 H_w(\lambda,\mu)=
 \prod_{i=1}^s
 \frac{\lambda^{\alpha_i}-\lambda^{-\alpha_i}}{\lambda-\lambda^{-1}}
 \prod_{i=1}^s
 \frac{\mu^{\beta_i}-\mu^{-\beta_i}}{\mu-\mu^{-1}}.
\end{equation}
Moreover, if
\[
 f_w(x,y,z)=\sum_{k=0}^d g_k(x,y)z^k,
 \qquad g_k\in\Z[x,y],
\]
then under the specialization homomorphism
\[
 \Psi: \Z[x,y,z]\longrightarrow K[t],
\]
\[
 \Psi(x)=\lambda+\lambda^{-1},
 \qquad
 \Psi(y)=\mu+\mu^{-1},
 \qquad
 \Psi(z)=\lambda\mu+\lambda^{-1}\mu^{-1}+t^2,
\]
one has
\[
 \Psi(f_w)=\Tr(w(A,B)).
\]
Consequently $g_k=0$ for $k>s$, while $g_s\ne0$ and
\begin{equation}\label{eq:gs-specializes-horowitz}
 g_s(\lambda+\lambda^{-1},\mu+\mu^{-1})=H_w(\lambda,\mu).
\end{equation}
\end{lem}

\begin{proof}
First note that $A,B\in\SL(2,R)$, because $\det A=\det B=1$.  Since determinant-one
matrices over a commutative ring are invertible, evaluating any group word in $A,B$ is
well-defined in $\SL(2,R)$.  The identities
\[
 \Tr(A)=\lambda+\lambda^{-1},
 \quad
 \Tr(B)=\mu+\mu^{-1},
 \quad
 \Tr(AB)=\lambda\mu+\lambda^{-1}\mu^{-1}+t^2
\]
show that substituting these three expressions into the trace polynomial
$f_w$ as $x=\lambda+\lambda^{-1}$, $y=\mu+\mu^{-1}$, and
$z=\lambda\mu+\lambda^{-1}\mu^{-1}+t^2$ gives $\Tr(w(A,B))$:
\[
f_w(\lambda+\lambda^{-1}, \mu+\mu^{-1}, \lambda\mu+\lambda^{-1}\mu^{-1}+t^2)=\Tr(w(A,B)).
\]

For $m\ne0$, a direct computation gives
\[
 A^m=
 \begin{pmatrix}
  \lambda^m & t\,\dfrac{\lambda^m-\lambda^{-m}}{\lambda-\lambda^{-1}}\\
  0&\lambda^{-m}
 \end{pmatrix},
\]
and
\[
 B^m=
 \begin{pmatrix}
  \mu^m&0\\
  t\,\dfrac{\mu^m-\mu^{-m}}{\mu-\mu^{-1}}&\mu^{-m}
 \end{pmatrix}.
\]
In the product
\[
 M=A^{\alpha_1}B^{\beta_1}\cdots A^{\alpha_s}B^{\beta_s},
\]
for computing any entry $M_{ij}$ of $M$, each factor $A^{\alpha_i}$ contributes at most one power of $t$, and each factor
$B^{\beta_i}$ contributes at most one power of $t$.  Hence the $t$-degree of $M_{ij}$ is at most
$2s$, and the same holds for $\Tr(M)$. Each diagonal entry has $t$-degree $0$, while each off-diagonal entry has $t$-degree exactly $1$.  Hence the only way to obtain a term of degree $2s$ in $\Tr(M)=M_{11}+M_{22}$ is to choose the off-diagonal
entry from every one of these $2s$ factors.  The resulting matrix-unit product is
\[
 E_{12}E_{21}E_{12}E_{21}\cdots E_{12}E_{21}=E_{11},
\]
so it contributes to the trace, and its $t$-coefficient in $\Tr(M)$ is exactly
\eqref{eq:horowitz-leading-precise}.  This coefficient is nonzero in $K$, so the
$t$-degree is exactly $2s$.

Now write $f_w=\sum_k g_k(x,y)z^k$.  Since $\Psi(z)=z_0+t^2$, where
$z_0=\lambda\mu+\lambda^{-1}\mu^{-1}$, the contribution of a nonzero term
$g_k(x,y)z^k$ has $t$-degree $2k$ after applying $\Psi$, unless
$g_k(\lambda+\lambda^{-1},\mu+\mu^{-1})=0$.  The map
\[
 \Z[x,y]\longrightarrow \Z[\lambda^{\pm1},\mu^{\pm1}],
 \qquad
 x\mapsto\lambda+\lambda^{-1},\quad
 y\mapsto\mu+\mu^{-1},
\]
is injective: for example, if a nonzero polynomial in $x,y$ has highest total term
with a lexicographically maximal exponent pair $(p,q)$, then the Laurent monomial
$\lambda^p\mu^q$ cannot be cancelled by the images of smaller exponent pairs.
Therefore $g_k(\lambda+\lambda^{-1},\mu+\mu^{-1})=0$ implies $g_k=0$.
Since $\Psi(f_w)$ has $t$-degree exactly $2s$, it follows that $g_k=0$ for all
$k>s$.  Comparing the coefficient of $t^{2s}$ gives
\eqref{eq:gs-specializes-horowitz}, and in particular $g_s\ne 0$.
\end{proof}

\begin{nt}\label{nt:cyclic-pair-counts}
For the rest of this section we use the letters $a,a^{-1},b,b^{-1}$ explicitly.  If
$w=u_1\cdots u_n$ is a cyclically reduced word over $\{a,a^{-1},b,b^{-1}\}$, indices are read
modulo $n$.  For $r,s\in\{a,a^{-1},b,b^{-1}\}$ put
\[
 N_{rs}(w)=\#\{i:u_i=r,\ u_{i+1}=s\}.
\]
Also put
\[
 m_a(w)=\#\{i:u_i\in\{a,a^{-1}\}\},\qquad
 m_b(w)=\#\{i:u_i\in\{b,b^{-1}\}\},
\]
and
\[
 R(w)=N_{ab}(w)+N_{b^{-1}a^{-1}}(w).
\]
\end{nt}

\begin{lem}\label{lem:type-change-count-balance}
For every cyclic word $w$ over $\{a,a^{-1},b,b^{-1}\}$,
\[
 N_{ab}(w)+N_{b^{-1}a^{-1}}(w)=N_{ba}(w)+N_{a^{-1}b^{-1}}(w).
\]
Consequently
\[
 R(w)=\frac{1}{2}\bigl(N_{ab}(w)+N_{ba}(w)+N_{a^{-1}b^{-1}}(w)+N_{b^{-1}a^{-1}}(w)\bigr).
\]
\end{lem}

\begin{proof}
Traverse the cyclic word and record only those cyclic adjacencies at which the
underlying generator changes from the $a$-type letters $\{a,a^{-1}\}$ to the $b$-type
letters $\{b,b^{-1}\}$, or conversely.  The $a$-to-$b$ changes and the $b$-to-$a$
changes alternate around the circle.  Put weight $+1$ on the letters $a$ and $b^{-1}$, and
weight $-1$ on the letters $a^{-1}$ and $b$.  The adjacencies counted by
$N_{ab}+N_{b^{-1}a^{-1}}$ are precisely the type changes from weight $+1$ to weight $-1$, while
those counted by $N_{ba}+N_{a^{-1}b^{-1}}$ are precisely the type changes from weight $-1$ to
weight $+1$.  Around a circle the number of such changes in the two directions is the
same.  This gives the stated identity.
\end{proof}

\begin{prop}[Top homogeneous term and exact degree]\label{prop:top-homogeneous-exact-degree}
Let $w=u_1\cdots u_n\in F(a,b)$ be a nontrivial cyclically reduced word over
$\{a,a^{-1},b,b^{-1}\}$.  Then
\begin{equation}\label{eq:top-homogeneous-general}
 [f_w]_{\rm top}=(-1)^{N_{ab}(w)+N_{ba}(w)}
 x^{m_a(w)-R(w)}y^{m_b(w)-R(w)}z^{R(w)}.
\end{equation}
In particular,
\begin{equation}\label{eq:exact-degree-general}
 \deg f_w=n-R(w)
 =n-\frac{1}{2}\bigl(N_{ab}(w)+N_{ba}(w)+N_{a^{-1}b^{-1}}(w)+N_{b^{-1}a^{-1}}(w)\bigr).
\end{equation}
\end{prop}

\begin{proof}
Let $K=\mathbb Q(X,Y,Z)$ and work in the Laurent-series field $K((t^{-1}))$.  Choose
$q\in K((t^{-1}))$ satisfying
\[
 q+q^{-1}=Zt,
\]
with leading term $q=Zt+O(t^{-1})$.  Hence $q$ has $t$-degree $1$ and leading
coefficient $Z$, while $q^{-1}$ has $t$-degree $-1$.  Put
\[
 {\mathsf A}_t=\begin{pmatrix}Xt&-1\\ 1&0\end{pmatrix},\qquad
 {\mathsf B}_t=\begin{pmatrix}0&q^{-1}\\ -q&Yt\end{pmatrix}.
\]
Then ${\mathsf A}_t,{\mathsf B}_t\in\SL(2,K((t^{-1})))$ and
\[
 \Tr {\mathsf A}_t=Xt,\qquad
 \Tr {\mathsf B}_t=Yt,\qquad
 \Tr({\mathsf A}_t{\mathsf B}_t)=Zt.
\]
Therefore
\begin{equation}\label{eq:specialization-degree-top}
 \Tr w({\mathsf A}_t,{\mathsf B}_t)=f_w(Xt,Yt,Zt).
\end{equation}
The highest power of $t$ in the right-hand side is the total degree of $f_w$, and
its coefficient is the top homogeneous part $[f_w]_{\rm top}(X,Y,Z)$.

We now compute the leading term of the left-hand side.  The inverses are
\[
 {\mathsf A}_t^{-1}=\begin{pmatrix}0&1\\ -1&Xt\end{pmatrix},\qquad
 {\mathsf B}_t^{-1}=\begin{pmatrix}Yt&-q^{-1}\\ q&0\end{pmatrix}.
\]
Expand the trace as a sum over closed two-state paths
\[
 \epsilon_0,\epsilon_1,\dots,\epsilon_n=\epsilon_0,\qquad
 \epsilon_i\in\{1,2\},
\]
where the $i$-th factor contributes the matrix entry from row $\epsilon_{i-1}$ to
column $\epsilon_i$.  We give each entry its $t$-degree and its leading coefficient.
The relevant leading data are
\[
\begin{array}{c|cccc}
 &11&12&21&22\\ \hline
 a&X&-1&1&0\\
 a^{-1}&0&1&-1&X\\
 b&0&0&-Z&Y\\
 b^{-1}&Y&0&Z&0
\end{array}
\]
where an entry $0$ means that no term of nonnegative $t$-degree is available in that
position; the displayed symbols are the leading coefficients of the nonnegative-degree
entries, and the entries $X,Y,Z$ all have $t$-degree $1$ while the entries $\pm1$ have
$t$-degree $0$.

We record the elementary two-state calculation needed from this table.  For a letter
$u$, let $E^+(u)$ be the set of possible exit states of degree-one entries in the row
of the table for $u$, and let $E^-(u)$ be the set of possible entry states of such
entries.  Thus
\[
\begin{array}{c|cccc}
u&a&a^{-1}&b&b^{-1}\\ \hline
E^-(u)&\{1\}&\{2\}&\{2\}&\{1,2\}\\
E^+(u)&\{1\}&\{2\}&\{1,2\}&\{1\}.
\end{array}
\]
A cyclic boundary $uv$ can have degree-one contributions from both adjacent letters
only if $E^+(u)\cap E^-(v)\ne\emptyset$.  Since $w$ is cyclically reduced, the only
cyclic adjacencies for which this intersection is empty are
\[
   ab\quad\hbox{and}\quad b^{-1}a^{-1}.
\]
For reference, the nonexceptional cyclically reduced adjacent pairs have the following
unique boundary states in $E^+(u)\cap E^-(v)$:
\[
\begin{array}{c|cccccccccc}
uv&aa&ab^{-1}&a^{-1}a^{-1}&a^{-1}b&a^{-1}b^{-1}&ba&ba^{-1}&bb&b^{-1}a&b^{-1}b^{-1}\\ \hline
\hbox{state}&1&1&2&2&2&1&2&2&1&1
\end{array}
\]
Hence every closed path has to lose at least one unit of $t$-degree at each of the
$R(w)=N_{ab}(w)+N_{b^{-1}a^{-1}}(w)$ such boundaries.  These losses are independent in the following elementary sense: two such exceptional boundaries cannot be consecutive in a cyclically reduced word, since a boundary following $ab$ begins with $b$, while a boundary following $b^{-1}a^{-1}$ begins with $a^{-1}$, and neither $b$ nor $a^{-1}$ can begin one of the exceptional pairs $ab,b^{-1}a^{-1}$.  Therefore every closed path has
$t$-degree at most
\begin{equation}\label{eq:maxplus-bound}
 n-R(w).
\end{equation}

This upper bound is attained as follows.  At every cyclic boundary $uv$ with
$E^+(u)\cap E^-(v)\ne\emptyset$, choose the boundary state in that intersection; for a cyclically reduced adjacent pair this intersection is a singleton.  At a
boundary $ab$ choose the boundary state $2$; then the preceding $a$ uses the
zero-degree entry $a_{12}=-1$, while the following $b$ can still use a degree-one
entry.  At a boundary $b^{-1}a^{-1}$ choose the boundary state $1$; then the following $a^{-1}$ uses
the zero-degree row-$1$, column-$2$ entry, while the preceding $b^{-1}$ can still use a degree-one
entry.  These local choices are compatible around the cycle, because they prescribe
exactly one boundary state at each cyclic boundary.  More explicitly, the prescribed
state at the boundary between the $i$-th and $(i+1)$-st letters determines at the same
time the exit state of the $i$-th letter and the entry state of the $(i+1)$-st letter.
Since no boundary is prescribed twice and the exceptional boundaries are not
consecutive, the prescriptions do not conflict.  Conversely, any closed path of
$t$-degree $n-R(w)$ must lose exactly one degree at each exceptional boundary and none
elsewhere.  Therefore at every nonexceptional boundary its state must lie in the
singleton $E^+(u)\cap E^-(v)$, while at an $ab$ boundary it must be $2$ and at a $b^{-1}a^{-1}$
boundary it must be $1$.  Hence the maximal-degree closed path is unique.  In
particular these prescriptions determine a unique closed path of $t$-degree $n-R(w)$.

It remains only to identify the leading coefficient.  In the unique maximal path,
each boundary $ab$ or $b^{-1}a^{-1}$ accounts for exactly one loss of degree and no other
boundary accounts for a loss.  Reading the boundary states from the preceding table
shows the following.  Among the $a$-type letters, precisely the letters immediately
preceding a cyclic adjacency $ab$ and the letters immediately following a cyclic
adjacency $b^{-1}a^{-1}$ use a zero-degree off-diagonal entry; all other $a$-type letters
contribute an $X$.  Hence the exponent of $X$ is $m_a(w)-R(w)$.  Among the
$b$-type letters, precisely the letters immediately preceding a cyclic adjacency
$ba$ and the letters immediately following a cyclic adjacency $a^{-1}b^{-1}$ contribute a
$Z$; all other $b$-type letters contribute a $Y$.  By
\cref{lem:type-change-count-balance}, the number of such $Z$-contributions is
$N_{ba}(w)+N_{a^{-1}b^{-1}}(w)=R(w)$, and therefore the exponent of $Y$ is $m_b(w)-R(w)$.
The only negative signs in the unique maximal path occur at the zero-degree entry
$a_{12}=-1$ attached to a cyclic adjacency $ab$, and at the leading entry
$b_{21}=-Z$ attached to a cyclic adjacency $ba$.  The entries attached to $b^{-1}a^{-1}$ and
$a^{-1}b^{-1}$ have positive leading coefficient.  Therefore the coefficient of $t^{n-R(w)}$
in the closed-path sum is
\begin{equation}\label{eq:maxplus-leading-coefficient}
 (-1)^{N_{ab}(w)+N_{ba}(w)}
 X^{m_a(w)-R(w)}Y^{m_b(w)-R(w)}Z^{R(w)}.
\end{equation}
In particular this coefficient is nonzero.  This proves both \eqref{eq:maxplus-bound}
and \eqref{eq:maxplus-leading-coefficient}.

Combining \eqref{eq:specialization-degree-top} with
\eqref{eq:maxplus-leading-coefficient} gives \eqref{eq:top-homogeneous-general}.  The
monomial in \eqref{eq:top-homogeneous-general} is nonzero, so its total degree is
\[
 (m_a-R)+(m_b-R)+R=n-R.
\]
This proves the first equality in \eqref{eq:exact-degree-general}, and the second follows
from \cref{lem:type-change-count-balance}.
\end{proof}

\smallskip
\noindent For example, for the commutator word \(w=aba^{-1}b^{-1}\), one has \(R(w)=1\) and
\(m_a(w)=m_b(w)=2\), so \eqref{eq:top-homogeneous-general} gives
\([f_w]_{\rm top}=-xyz\), in agreement with
\(f_{[a,b]}=x^2+y^2+z^2-xyz-2\).
\smallskip

\begin{thm}[Sharp degree bounds]\label{thm:degree-lower}
For every nontrivial cyclically reduced word $w$ of length $n$,
\[
 \left\lceil \frac{n}{2}\right\rceil\le \deg f_w\le n.
\]
Moreover, for every $n\ge 1$, both the lower and upper bounds are attained by some cyclically reduced word of length $n$.
\end{thm}

\begin{proof}
The upper bound follows immediately from \cref{prop:top-homogeneous-exact-degree}, since
$\deg f_w=n-R(w)$ and $R(w)\ge0$.  It is attained, for every $n\ge1$, by the generator power
$a^n$, because the Chebyshev trace polynomial $f_{a^n}(x)=2T_n(x/2)$ has degree $n$.

If $w$ is a power of a single generator, the lower bound follows from the same Chebyshev trace
polynomial.

Otherwise write $w$ cyclically as
\begin{equation}\label{eq:syllable-word-degree-lower}
 w=a^{\alpha_1}b^{\beta_1}\cdots a^{\alpha_s}b^{\beta_s},
\end{equation}
with all $\alpha_i,\beta_i\ne0$.  Let
\[
 n=||w||_{\mathcal X}=\sum_i |\alpha_i|+\sum_i |\beta_i|.
\]
By \cref{prop:top-homogeneous-exact-degree},
\[
 \deg f_w=n-R(w).
\]
The number $R(w)=N_{ab}(w)+N_{b^{-1}a^{-1}}(w)$ counts selected type-changing adjacencies, so
$R(w)\le s$.  Since $s\le n/2$, this implies
\[
 \deg f_w\ge n-s\ge \left\lceil \frac{n}{2}\right\rceil.
\]
Sharpness of this lower bound can be seen directly from \cref{prop:top-homogeneous-exact-degree}.  For even $n=2k$, take
\[
 w=(ab)^k.
\]
Then $R(w)=k$, and hence $\deg f_w=2k-k=k$.  For odd $n=2k+1$, take
\[
 w=(ab)^k a.
\]
Then again $R(w)=k$, and hence
\[
 \deg f_w=2k+1-k=k+1=\left\lceil n/2\right\rceil.
\]
\end{proof}

\section{Positive words and typical degree}

A \emph{positive word} is a (nontrivial) word in the monoid generated by $a,b$.

\begin{prop}[Degree of positive words]\label{prop:positive-degree}
For a positive word 
\[
 w=a^{\alpha_1}b^{\beta_1}\cdots a^{\alpha_s}b^{\beta_s},
\]
in $F(a,b)$ in standard syllable form involving both $a$ and $b$ as above, with
\[
 n=||w||_{\mathcal X}=\sum_i\alpha_i+\sum_i\beta_i,
\]
and $\alpha_i,\beta_i>0$, one has
\[
 \deg f_w=n-s,
\]
where $2s$ is the syllable length of $w$.
\end{prop}

\begin{proof}
For a positive cyclic word in standard syllable form, the cyclic adjacent pair $ab$
occurs exactly once at the beginning of each $b$-syllable, and no adjacent pair of type
$b^{-1}a^{-1}$ occurs.  Thus $R(w)=N_{ab}(w)+N_{b^{-1}a^{-1}}(w)=s$.  The exact formula
\eqref{eq:exact-degree-general} gives $\deg f_w=n-s$.
\end{proof}

\begin{ex}\label{ex:positive-degree-example}
For instance, the positive word
\[
 w=a^2b^3ab=a^2b^3a^1b^1
\]
has length \(n=7\) and has \(s=2\) alternating \(a\)- and \(b\)-syllable pairs.  Hence \cref{prop:positive-degree} gives
\[
 \deg f_w=n-s=5.
\]
This illustrates that, for positive words involving both generators, the degree is obtained from the ordinary length by subtracting the number of cyclic \(a\)-syllables, equivalently the number of cyclic \(b\)-syllables.
\end{ex}

\begin{rem}\label{rem:positive-proof-caution}
Thus the positive-word degree formula is now a direct specialization of the general
leading-term theorem.  In the positive case the number subtracted from the length is
exactly the number of cyclic occurrences of the transition from an $a$-block to a
$b$-block, equivalently the number of cyclic $a$-blocks and the number of cyclic
$b$-blocks.
\end{rem}

\begin{rem}\label{rem:OA-notation}
In the probabilistic estimates below the notation $O_D(\cdot)$ means that the
implicit multiplicative constant is allowed to depend on the fixed parameter $D>0$, but not on
$n$.  Thus a statement such as
\[
 X_n=O_D(\sqrt{n\log n})
 \quad\text{with probability at least }1-O_D(n^{-D})
\]
means that for every fixed $D>0$ there exist constants $C_D,C'_D>0$ such that,
for all sufficiently large $n$,
\[
 \Pr\bigl(|X_n|\le C_D\sqrt{n\log n}\bigr)
 \ge 1-C'_D n^{-D}.
\]
\end{rem}

The following is the standard bounded-differences inequality of McDiarmid; see McDiarmid's survey \cite[Section 3]{McDiarmid1989}.  We state only the special form needed below.

\begin{prop}[McDiarmid's bounded-differences inequality, special form]\label{prop:mcdiarmid-special}
Let $X_1,\ldots,X_n$ be independent random variables taking values in arbitrary sets, and let
\[
 F=F(X_1,\ldots,X_n)
\]
be a real-valued function of these variables.  Suppose that changing one coordinate can change the value of $F$ by at most $c_j$ in the following precise sense: for every $j$ and for every two inputs
\[
 (x_1,\ldots,x_n),\qquad (x'_1,\ldots,x'_n)
\]
which agree in all coordinates except possibly the $j$-th coordinate, one has
\[
 \bigl|F(x_1,\ldots,x_n)-F(x'_1,\ldots,x'_n)\bigr|\le c_j.
\]
Then for every $t>0$,
\[
 \Pr\bigl(|F-\mathbb E F|\ge t\bigr)
 \le
 2\exp\left(-\frac{2t^2}{\sum_{j=1}^n c_j^2}\right).
\]
In particular, if $c_j\le 2$ for every $j$, then
\[
 \Pr\bigl(|F-\mathbb E F|\ge t\bigr)
 \le
 2\exp\left(-\frac{t^2}{2n}\right).
\]
\end{prop}

\begin{cor}[Random positive words]\label{cor:random-positive-degree}
Let $0<p<1$, and let $w_n$ be a random positive word of length $n$ in the alphabet $\{a,b\}$, with independent letters satisfying
\[
 \Pr(a)=p,\qquad \Pr(b)=1-p.
\]
Then, for every $D>0$,  with probability at least $1-O_D(n^{-D})$ as $n\to\infty$ we have
\[
 \deg f_{w_n}=\bigl(1-p(1-p)\bigr)n+O_D(\sqrt{n\log n}).
\]
In particular, in the unbiased case $p=1/2$,
\[
 \deg f_{w_n}=\frac{3n}{4}+O_D(\sqrt{n\log n})
\]
with probability at least $1-O_D(n^{-D})$.
\end{cor}

\begin{proof}
Let 
\[
w_n=\xi_1\dots \xi_n
\]
be a random positive word of length $n$, where $\xi_1,\dots, \xi_n$ are i.i.d. random variables with values in $\{a,b\}$ and $\Pr(a)=p, \Pr(b)=1-p$.

Let $T_n$ be the number of cyclic transitions between $a$ and $b$ in $w_n$.
Since $0<p<1$, the cases $T_n=0$, where $w_n=a^n$ or $w_n=b^n$, have exponentially small probability in $n$. Thus we assume that $T_n>0$. 
Then $T_n$ is the syllable length of any cyclic permutation $w_n'$ of $w_n$ in standard syllable form, and
\[
 s=T_n/2
\]
is the number of $a$-power syllables and also the number of $b$-power syllables in $w_n'$.
Note that $T_n$ is exactly the number of indices $i$, counted cyclically, such that $\xi_i\ne \xi_{i+1}$. By independence, for each $i$ (counted cyclically), $\Pr(\xi_i\ne \xi_{i+1})=2p(1-p)$. Hence
\[
 \mathbb E[T_n]=2p(1-p)n.
\]

To obtain a convergence speed estimate we will use McDiarmid's bounded-differences inequality.

 Regard \(T_n\) as a function
\[
 F(\xi_1,\ldots,\xi_n)=T_n
\]
of the independent letters.  If one replaces the \(j\)-th letter \(\xi_j\) by  \(\xi_j'\), then at most two cyclic transitions adjacent to this letter can be affected in the computation of $T_n$.  Consequently the bounded-difference constant for each coordinate is at most \(2\):
\[
 |F(\xi_1,\ldots,\xi_j,\ldots,\xi_n)-F(\xi_1,\ldots,\xi'_j,\ldots,\xi_n)|\le 2.
\]
McDiarmid's inequality therefore gives, for every \(t>0\),
\[
 \Pr\left(|T_n-\mathbb E T_n|\ge t\right)
 \le
 2\exp\left(-\frac{2t^2}{\sum_{j=1}^n 2^2}\right)
 =
 2\exp\left(-\frac{t^2}{2n}\right).
\]
In particular, after choosing the constant \(C_D\) sufficiently large and taking
\[
 t=C_D\sqrt{n\log n},
\]
we get
\[
 \Pr\left(|T_n-2p(1-p)n|\ge C_D\sqrt{n\log n}\right)=O_D(n^{-D}).
\]

It remains to translate this transition estimate into a degree estimate.  If the cyclic word contains both letters, then the number of cyclic \(a\)-blocks equals the number of cyclic \(b\)-blocks and is exactly
\[
 s=T_n/2.
\]
If the word is all \(a\)'s or all \(b\)'s, then \(T_n=0\), and \cref{lem:chebyshev-power} gives \(\deg f_{w_n}=n\), which is again \(n-T_n/2\).  Thus in every case
\[
 \deg f_{w_n}=n-\frac{T_n}{2}.
\]
Combining this identity with the concentration estimate for \(T_n\) yields
\[
 \deg f_{w_n}
 =n-p(1-p)n+O_D(\sqrt{n\log n})
 =\bigl(1-p(1-p)\bigr)n+O_D(\sqrt{n\log n})
\]
with probability at least \(1-O_D(n^{-D})\), as claimed.

\end{proof}

\begin{conv}
For a freely reduced word $w\in F(a,b)$, let $\tau(w)$ be the unique maximal word
such that
\[
w=\tau(w) w^{\cyc}\tau(w)^{-1},
\]
where the product is freely reduced as written and where $w^{\cyc}$ is cyclically
reduced. We call $w^{\cyc}$ the \emph{cyclically reduced form of $w$.}
\end{conv}

The following statement is well known but we include a proof here for completeness.

\begin{lem}[Cyclic cancellation in a random freely reduced word]\label{lem:cyclic-cancellation-random}
Let
\[
 W_n=s_1s_2\cdots s_n
\]
be a random freely reduced word in $F(a,b)$ generated by the standard nonbacktracking Markov chain on
\(\{a,a^{-1},b,b^{-1}\}\): the first letter is uniform and each next letter is chosen uniformly among the three letters which are not the inverse of the current one.  Let $K_n=|\tau(W_n)|$, so that 
\[
|W_n^{\cyc}|=n-2K_n.
\]

Then, for every \(k\ge1\), we have
\begin{equation}\label{eq:cyclic-cancellation-tail}
 \Pr(K_n\ge k)\le 3^{1-k}.
\end{equation}
Consequently, for every \(D>0\), there is a constant \(C_D>0\) such that
\[
 \Pr(K_n>C_D\log n)=O_D(n^{-D}),
\]
and \(\mathbb E K_n=O(1)\).
\end{lem}

\begin{proof}
The event \(K_n\ge k\) means that the first \(k\) letters and the last \(k\) letters match in inverse reverse order:
\[
 s_{n-j+1}=s_j^{-1}\qquad 1\le j\le k,
\]
up to the harmless convention that the estimate is trivial when \(2k>n\).  Once the first \(k\) letters and the middle part of $W_n$ are fixed, the last \(k\) letters are uniquely determined.  A prescribed reduced suffix of length \(k\) has conditional probability at most \(3^{-(k-1)}\): the first letter of the suffix has probability at most \(1\), and each subsequent letter has conditional probability at most \(1/3\).  This proves \eqref{eq:cyclic-cancellation-tail}.

Taking, for instance, \(C_D=D+3\) with logarithms to base \(3\) gives
\[
 \Pr(K_n>(D+3)\log_3 n)\le 3n^{-(D+3)}=O_D(n^{-D}).
\]
The expectation bound follows from the tail-sum formula for nonnegative integer-valued random variables:
\[
 \mathbb E K_n=\sum_{k=1}^\infty \Pr(K_n\ge k)
 \le \sum_{k=1}^\infty 3^{1-k}<\infty,
\]
with the final constant independent of \(n\).
\end{proof}

\begin{rem}[Relation to standard genericity estimates]\label{rem:cyclic-cancellation-literature}
The exponential tail estimate in Lemma~\ref{lem:cyclic-cancellation-random} is included for completeness, but it should be regarded as a standard elementary feature of the nonbacktracking random-walk model for freely reduced words.  Closely related exponentially small exceptional-set estimates occur throughout the generic-case complexity and genericity literature for free groups and hyperbolic groups; see, for example, Kapovich--Myasnikov--Schupp--Shpilrain \cite{KMSS2003}, Kapovich--Schupp--Shpilrain \cite{KSS2006}, and Gilman \cite{Gilman2010}.  The novelty of the present application, if any, is not the cancellation estimate itself, but its use here to pass from a random freely reduced linear word to its cyclically reduced core when estimating the degree of the trace polynomial.
\end{rem}

\begin{lem}[Transition-count concentration in the nonbacktracking model]\label{lem:directed-transition-concentration}
Let
\[
 W_n=s_1s_2\cdots s_n
\]
be generated by the standard nonbacktracking model on
\(\{a,a^{-1},b,b^{-1}\}\), with the first letter uniform and each next letter chosen uniformly
among the three letters not inverse to the current one.  Let \(E\) be any set of
allowed directed pairs of letters and put
\[
 T_E(W_n)=\#\{1\le i<n:(s_i,s_{i+1})\in E\}.
\]
Then
\[
 \mathbb E T_E(W_n)=\frac{|E|(n-1)}{12}.
\]
Moreover, there are absolute constants \(c,C>0\), independent of \(E\) and \(n\), such that
for every \(t>0\),
\begin{equation}\label{eq:directed-transition-concentration}
 \Pr\left(\left|T_E(W_n)-\frac{|E|(n-1)}{12}\right|\ge t\right)
 \le C\exp\left(-\frac{c t^2}{n}\right).
\end{equation}
Consequently, for every \(D>0\),
\[
 T_E(W_n)=\frac{|E|n}{12}+O_D(\sqrt{n\log n})
\]
with probability at least \(1-O_D(n^{-D})\).
\end{lem}

\begin{proof}
The expectation follows from stationarity.  The uniform distribution on
\(\{a,a^{-1},b,b^{-1}\}\) is stationary for the nonbacktracking transition matrix, and each
allowed directed pair therefore has probability \(1/4\cdot 1/3=1/12\) at each
linear position.

We next recall the elementary bounded-differences concentration estimate for this
finite uniformly mixing Markov chain; the following paragraph is a direct
Doob-martingale proof using an exponential coupling.  Let
\(X_1,X_2,\ldots\) be this four-state Markov chain and let
\[
 F=\sum_{i=1}^{n-1} g(X_i,X_{i+1}),
 \qquad 0\le g\le 1.
\]
For this chain the two-step transition matrix has all entries bounded below by a
positive absolute constant, say \(\delta>0\).  Hence two copies of the chain started
at different states can be coupled so that, after each two-step block before they
meet, they have probability at least \(4\delta\) of meeting, and after meeting they
move together.  Thus the probability that the two coupled chains still disagree
after \(2m\) further steps is at most \((1-4\delta)^m\).  It follows that changing
the state revealed at a given time can change the conditional expectation of the
future contribution to \(F\) by at most an absolute constant: the possible
discrepancy at distance \(k\) is bounded by a geometric sequence in
\(\lfloor k/2\rfloor\).  Thus the Doob martingale
\[
 M_j=\mathbb E(F\mid X_1,
 \ldots,X_j),\qquad 0\le j\le n,
\]
has increments bounded in absolute value by an absolute constant \(C_0\), independent
of \(n\) and of \(g\).  Azuma--Hoeffding gives
\[
 \Pr(|F-\mathbb E F|\ge t)
 \le 2\exp\left(-\frac{t^2}{2C_0^2 n}\right).
\]
Applying this with \(g\) equal to the indicator of \(E\) proves
\eqref{eq:directed-transition-concentration}.  The final high-probability form follows
by taking \(t=C_D\sqrt{n\log n}\) with \(C_D\) sufficiently large.
\end{proof}

\begin{prop}[Random freely reduced words]\label{prop:random-freely-reduced-degree}
Let $W_n$ be a random freely reduced word of length $n$ in $F(a,b)$ in the standard
nonbacktracking model, and let $W_n^{\cyc}$ be its cyclically reduced form.  Put
\[
 N_n=|W_n^{\cyc}|.
\]
Then for every $D>0$ there is a constant $C_D>0$ such that, with probability at least
$1-O_D(n^{-D})$,
\[
 N_n=n-O_D(\log n),
\]
and
\[
 \deg f_{W_n}=\deg f_{W_n^{\cyc}}
 =\frac{5}{6}N_n+O_D(\sqrt{n\log n}).
\]
Moreover, after possibly increasing $C_D$,
\[
 \deg f_{W_n}=\frac{5n}{6}+O_D(\sqrt{n\log n})
\]
with probability at least $1-O_D(n^{-D})$.
\end{prop}

\begin{proof}
The equality $\deg f_{W_n}=\deg f_{W_n^{\cyc}}$ holds because cyclic reduction
conjugates the represented group element and trace polynomials are conjugacy-invariant.
By \cref{lem:cyclic-cancellation-random}, with probability at least
$1-O_D(n^{-D})$ we have $N_n=n-O_D(\log n)$.

Let
\[
 R_n^{\rm lin}=\#\{1\le i<n:(s_i,s_{i+1})=(a,b)\text{ or }(b^{-1},a^{-1})\}
\]
for the original freely reduced word $W_n=s_1\cdots s_n$.  Applying
\cref{lem:directed-transition-concentration} to the two-element set
\[
 E=\{(a,b),(b^{-1},a^{-1})\}
\]
gives
\[
 \mathbb E R_n^{\rm lin}=\frac{n-1}{6}
\]
and, for every $D>0$,
\begin{equation}\label{eq:R-linear-concentration}
 R_n^{\rm lin}=\frac{n}{6}+O_D(\sqrt{n\log n})
\end{equation}
with probability at least $1-O_D(n^{-D})$.

Cyclic reduction removes $O_D(\log n)$ letters at the two ends with the same high
probability, and replacing the resulting linear word by the cyclic word changes the
count $R=N_{ab}+N_{b^{-1}a^{-1}}$ by at most $O_D(\log n)+1$.  Hence
\[
 R(W_n^{\cyc})=\frac{N_n}{6}+O_D(\sqrt{n\log n})
\]
with probability at least $1-O_D(n^{-D})$.  The exact degree formula
\eqref{eq:exact-degree-general} now gives
\[
 \deg f_{W_n}=\deg f_{W_n^{\cyc}}
 =N_n-R(W_n^{\cyc})
 =\frac{5}{6}N_n+O_D(\sqrt{n\log n}).
\]
Since $N_n=n-O_D(\log n)$, the final estimate follows after absorbing the logarithmic
term into $O_D(\sqrt{n\log n})$.
\end{proof}

\begin{cor}[Uniformly random cyclically reduced words]\label{cor:random-cyclically-reduced-degree}
Let $w_n$ be chosen uniformly from cyclically reduced words of length $n$ in $F(a,b)$.
Then for every $D>0$ there is a constant $C_D>0$ such that
\[
 \deg f_{w_n}=\frac{5n}{6}+O_D(\sqrt{n\log n})
\]
with probability at least $1-O_D(n^{-D})$.
\end{cor}

\begin{proof}
Choose a freely reduced word in the standard nonbacktracking model; this is the uniform distribution on freely reduced words of length $n$.  Condition on
the event that its last letter is not the inverse of its first letter.  This event is
exactly the event that the freely reduced word is cyclically reduced, and the conditional distribution is uniform on cyclically reduced words of length $n$.  Let $\mathcal R_n$ and $\mathcal C_n$ denote the sets of freely reduced and cyclically reduced words of length $n$, respectively.  Rivin's exact count \cite[Theorem~1.1]{Rivin2010}, specialized to rank two, gives
\[
 |\mathcal C_n|=3^n+2+(-1)^n,
 \qquad
 |\mathcal R_n|=4\cdot 3^{n-1}.
\]
Thus
\[
 \Pr(W_n\in\mathcal C_n)=\frac{|\mathcal C_n|}{|\mathcal R_n|}
 =\frac34+O(3^{-n}),
\]
and in particular this probability is bounded below by a positive absolute constant independent of $n$.  Conditioning therefore changes only the constants in the preceding concentration
estimates.  The count $R=N_{ab}+N_{b^{-1}a^{-1}}$ for the cyclic word differs from the linear
count only by the final last--first adjacency, so
\[
 R(w_n)=\frac{n}{6}+O_D(\sqrt{n\log n})
\]
with probability at least $1-O_D(n^{-D})$.  The exact formula
$\deg f_{w_n}=n-R(w_n)$ gives the claim.
\end{proof}

\begin{rem}\label{rem:empirical-degree}
The exact formula above explains the empirical ratios near $0.8$--$0.85$ observed in
small random freely reduced and cyclically reduced computations: for the uniform cyclically reduced model the
asymptotic degree ratio is $5/6$.
\end{rem}


\section{The constant term}

\begin{thm}[Constant term]\label{thm:constant-term}
For every word $w\in F(a,b)$,
\[
 f_w(0,0,0)\in\{-2,0,2\}.
\]
In particular, the constant term is uniformly bounded independently of $|w|_{\mathcal X}$.
\end{thm}

\begin{proof}
Let $w\in F(a,b)$ be an arbitrary freely reduced word.

We use the following explicit standard embedding of the quaternion group
\[
 Q_8=\{\pm 1,\pm \mathbf i,\pm \mathbf j,\pm \mathbf k\}
\]
into $\SL(2,\C)$.  Let
\[
 I_2=\begin{pmatrix}1&0\\0&1\end{pmatrix},\qquad
 \mathbf i=
 \begin{pmatrix}
  i&0\\
  0&-i
 \end{pmatrix},
 \qquad
 \mathbf j=
 \begin{pmatrix}
  0&1\\
  -1&0
 \end{pmatrix},
\]
and put
\[
 \mathbf k=\mathbf i\mathbf j=
 \begin{pmatrix}
  0&i\\
  i&0
 \end{pmatrix}.
\]
These matrices all have determinant $1$, and a direct multiplication gives
\[
 \mathbf i^2=\mathbf j^2=\mathbf k^2=-I_2,
 \qquad
 \mathbf i\mathbf j=\mathbf k=-\mathbf j\mathbf i.
\]
Thus the matrices $\{\pm I_2, \pm \mathbf i, \pm \mathbf j, \pm \mathbf k\}$  realize $Q_8$ as a subgroup of $\SL(2,\C)$.  Moreover
\[
 \Tr(\mathbf i)=\Tr(\mathbf j)=\Tr(\mathbf k)=0.
\]
In particular, for the homomorphism
\[
 \rho_0:F(a,b)\to \SL(2,\C),
 \qquad
 \rho_0(a)=\mathbf i,
 \qquad
 \rho_0(b)=\mathbf j,
\]
we have
\[
 \Tr(\rho_0(a))=0,
 \qquad
 \Tr(\rho_0(b))=0,
 \qquad
 \Tr(\rho_0(ab))=\Tr(\mathbf i\mathbf j)=\Tr(\mathbf k)=0.
\]
Therefore, by the defining property of the trace polynomial,
\[
 \Tr(\rho_0(w))=f_w(\Tr \mathbf i,\Tr \mathbf j,\Tr(\mathbf i\mathbf j))=f_w(0,0,0).
\]
The image of $\rho_0$ is contained in the eight-element subgroup
\[
 \{\pm I_2,\pm\mathbf i,\pm\mathbf j,\pm\mathbf k\}.
\]
For these matrices,
\[
 \Tr(\pm I_2)=\pm2,
 \qquad
 \Tr(\pm\mathbf i)=\Tr(\pm\mathbf j)=\Tr(\pm\mathbf k)=0.
\]
Hence $f_w(0,0,0)\in\{-2,0,2\}$.
\end{proof}

\section{Coefficient growth}\label{sec:coefficient-growth}

Recall that, using the notations from the Introduction,  for $w\in F(a,b)$
\[
\norm{f_w}_\infty=\max_{i,j,k}\left|[x^i y^j z^k]f_w(x,y,z)\right|
\]
and
\[
 \norm{f_w}_1=\sum_{i,j,k}\left|[x^i y^j z^k]f_w(x,y,z)\right|.
\]

We also denote 
\[
\varphi=\frac{1+\sqrt5}{2},
\]
the golden ratio, and put
\[
 \varphi_0=\sqrt{2\varphi^2+1}.
\]
Thus
\[
1.6<\varphi<\varphi_0<2.5.
\]

\begin{thm}[A universal exponential upper bound]\label{thm:coefficient-upper}
Let \(\varphi_0=\sqrt{2\varphi^2+1}\).  For every \(w\in F(a,b)\), put
\(n=\|w\|_{\mathcal X}\).  Then
\[
 \norm{f_w}_\infty\le 2\varphi_0^n
\]
and
\[
 \norm{f_w}_1\le 2\binom{n+3}{3}\varphi_0^n.
\]
\end{thm}

\begin{proof}
Since trace polynomials are conjugacy-invariant, we may replace \(w\) by a cyclically
reduced conjugate and assume that \(|w|_{\mathcal X}=n\).
Let $(x,y,z)\in(\C^*)^3$ satisfy
\[
 |x|=|y|=|z|=1.
\]
We first choose explicit matrices $A,B\in\SL(2,\C)$ whose three basic traces are exactly $x,y,z$.  Set
\[
 A=\begin{pmatrix}x&-1\\1&0\end{pmatrix}.
\]
Then $\det A=1$ and $\Tr A=x$.  Next choose $q,r\in\C$ satisfying
\[
 q-r=z,
 \qquad
 qr=-1.
\]
Equivalently, $r$ is a root of
\[
 r^2+zr+1=0
\]
and then $q=z+r$.  Put
\[
 B=\begin{pmatrix}0&q\\r&y\end{pmatrix}.
\]
Then
\[
 \det B=-qr=1,
 \qquad
 \Tr B=y.
\]
Moreover
\[
 AB=\begin{pmatrix}-r&xy-q\\0&q\end{pmatrix},
\]
so
\[
 \Tr(AB)=q-r=z.
\]
Therefore, by the defining property of the trace polynomial,
\[
 f_w(x,y,z)=\Tr(w(A,B))
\]
for this particular pair $A,B$.

We now estimate this trace.  Since $|z|=1$ and $r$ satisfies $r^2+zr+1=0$, the quadratic formula gives
\[
 r=\frac{-z\pm\sqrt{z^2-4}}2.
\]
Using $|z|=1$ and $|z^2-4|\le 5$, one obtains the crude bound
\[
 |r|\le \frac{1+\sqrt5}{2}=\varphi.
\]
Since $q=z+r$, the number $-q$ is the other root of $u^2+zu+1=0$.  The same bound therefore gives
\[
 |q|,|r|\le\varphi.
\]
We use the Frobenius norm
\[
 \|M\|_F=\left(\sum_{i,j}|m_{ij}|^2\right)^{1/2}.
\]
It is submultiplicative:
\[
 \|MN\|_F\le \|M\|_F\|N\|_F.
\]
For our matrix $A$ we have
\[
 A^{-1}=\begin{pmatrix}0&1\\-1&x\end{pmatrix},
\]
and hence, since $|x|=1$,
\[
 \|A\|_F=\|A^{-1}\|_F=\sqrt{|x|^2+1+1}=\sqrt3.
\]
For $B$ one has
\[
 B^{-1}=\begin{pmatrix}y&-q\\-r&0\end{pmatrix},
\]
and therefore, since $|y|=1$ and $|q|,|r|\le\varphi$,
\[
 \|B\|_F,\ \|B^{-1}\|_F
 \le \sqrt{1+2\varphi^2}=\varphi_0.
\]
Also $\sqrt{3}<\varphi_0$, because $\varphi_0^2=1+2\varphi^2>3$.  Hence every letter $a^{\pm1},b^{\pm1}$ in the word $w$ is evaluated by a matrix of Frobenius norm at most $\varphi_0$.  If $|w|=n$, submultiplicativity gives
\[
 \|w(A,B)\|_F\le \varphi_0^n.
\]
Finally, for a $2\times2$ matrix $M$,
\[
 |\Tr M|=|m_{11}+m_{22}|\le |m_{11}|+|m_{22}|\le \sqrt2\,(|m_{11}|^2+|m_{22}|^2)^{1/2}\le 2\|M\|_F.
\]
Thus, on the unit three-torus,
\begin{equation}\label{eq:torus-bound-for-cauchy}
 |f_w(x,y,z)|=|\Tr(w(A,B))|\le 2\varphi_0^n.
\end{equation}

We now pass from the uniform bound \eqref{eq:torus-bound-for-cauchy} to a coefficient bound.  Write
\[
 f_w(x,y,z)=\sum_{i,j,k\ge0} c_{ijk}x^iy^jz^k.
\]
Cauchy's coefficient formula in three variables says that
\[
 c_{ijk}=\frac{1}{(2\pi i)^3}
 \int_{|\xi|=1}\int_{|\eta|=1}\int_{|\zeta|=1}
 \frac{f_w(\xi,\eta,\zeta)}{\xi^{i+1}\eta^{j+1}\zeta^{k+1}}
 \,d\zeta\,d\eta\,d\xi.
\]
Since $|\xi|=|\eta|=|\zeta|=1$, the denominator has absolute value $1$.  Taking absolute values and using \eqref{eq:torus-bound-for-cauchy} gives
\[
 |c_{ijk}|\le \sup_{|\xi|=|\eta|=|\zeta|=1}|f_w(\xi,\eta,\zeta)|\le 2\varphi_0^n.
\]
Therefore
\[
 \norm{f_w}_\infty\le 2\varphi_0^n.
\]
By Proposition~\ref{prop:degree-upper}, all monomials in $f_w$ have total degree at most $n$, so the number of possible monomials is at most
\[
 \binom{n+3}{3}.
\]
Consequently
\[
 \norm{f_w}_1\le 2\binom{n+3}{3}\varphi_0^n.
\]
This completes the proof.
\end{proof}

\begin{thm}[Exponential lower bounds]\label{thm:coefficient-lower}
Put $w_n=a^n$ for all $n\ge 1$.  Then
\[
 \norm{f_{w_n}}_\infty=\Theta\left(\frac{\varphi^n}{\sqrt n}\right)
\]
 and
\[
  \norm{f_{w_n}}_1\sim\varphi^n, \quad \text{ that is } \quad 
\lim_{n\to\infty}\frac{\norm{f_{w_n}}_1}{\varphi^n}=1.
\]
\end{thm}

\begin{proof}
Put $w_n=a^n$ for $n\ge 1$.  Then
\[
 f_{a^n}(x)=2T_n(x/2)=:P_n(x),
\]
where $T_n$ is the Chebyshev polynomial of the first kind.  The polynomials $P_n$ of degree $n$ satisfy
\[
 P_0=2,
 \qquad
 P_1=x,
 \qquad
 P_{n+1}=xP_n-P_{n-1}.
\]

Equivalently,
\begin{equation}\label{eq:chebyshev-coeff-explicit}
 P_n(x)=\sum_{k=0}^{\lfloor n/2\rfloor}
 (-1)^k c_{n,k}x^{n-2k},
 \qquad
 c_{n,k}=\frac{n}{n-k}\binom{n-k}{k}.
\end{equation}
Thus the nonzero coefficients of $P_n$ are, up to signs, the positive integers $c_{n,k}$ for $k=0,\dots, \lfloor n/2\rfloor$. 
Their signs alternate as $k$ changes, so the $\ell^1$-norm is
\[
 \|P_n\|_1=\sum_{k=0}^{\lfloor n/2\rfloor}c_{n,k}.
\]
Equivalently, the explicit formula above shows that these sums are the absolute-value sums of the coefficients of $P_n$; the alternating signs prevent any cancellation in the $\ell^1$-norm.  The recurrence $P_{n+1}=xP_n-P_{n-1}$ then implies that these sums satisfy the Lucas numbers recurrence, and hence
\[
\norm{f_{a^n}}_1= \|P_n\|_1=L_n\sim \varphi^n,
\]
where $L_n=\varphi^n+(-1)^n\varphi^{-n}$ is the $n$-th Lucas number.

It remains to explain the size of the largest coefficient.  We first locate where the maximum occurs.  From \eqref{eq:chebyshev-coeff-explicit}, for $0\le k<\lfloor n/2\rfloor$,
\begin{equation}\label{eq:ratio-coefficients}
 \frac{c_{n,k+1}}{c_{n,k}}
 =\frac{(n-2k)(n-2k-1)}{(n-k-1)(k+1)}=\frac{(\frac{n}{k}-2)(\frac{n}{k}-2-\frac{1}{k})}{(\frac{n}{k}-1-\frac{1}{k})(1+\frac{1}{k})}
\end{equation}
Consequently the sequence $c_{n,k}$ increases while the right hand side of
\eqref{eq:ratio-coefficients} is greater than $1$ and decreases after it is less than
$1$.  If $k/n$ tends to a limit $\alpha\in [0,1/2]$, the transition from increasing to decreasing occurs asymptotically when 
\[
 (1-2\alpha)^2=\alpha(1-\alpha).
\]
The solution, in the interval $[0,1/2]$, is
\begin{equation}\label{eq:alpha0}
 \alpha_0=\frac{5-\sqrt5}{10}.
\end{equation}
Thus every maximizing index $k_n$ satisfies
\[
 \frac{k_n}{n}\longrightarrow \alpha_0.
\]
Equivalently, the largest coefficient occurs for $k$ equal to one of the nearest
integers to $\alpha_0 n$, up to a uniformly bounded error.

Now apply Stirling's formula to $c_{n,k}$ with $k=\lfloor \alpha n\rfloor$ and
$0<\alpha<1/2$.  Since
\[
 c_{n,k}=\frac{n}{n-k}\frac{(n-k)!}{k!(n-2k)!},
\]
Stirling's formula gives, uniformly for $\alpha$ in compact subintervals of
$(0,1/2)$,
\begin{equation}\label{eq:stirling-c-nk}
 c_{n,\lfloor\alpha n\rfloor}
 =\Theta\left(n^{-1/2}\exp(nF(\alpha))\right),
\end{equation}
where
\begin{equation}\label{eq:F-alpha}
 F(\alpha)=(1-\alpha)\log(1-\alpha)-\alpha\log\alpha
 -(1-2\alpha)\log(1-2\alpha).
\end{equation}
Indeed, \eqref{eq:stirling-c-nk} is obtained by substituting
$m!\sim \sqrt{2\pi m}(m/e)^m$ into
$(n-k)!/(k!(n-2k)!)$; the factor $n/(n-k)$ only contributes a bounded nonzero
factor.  Differentiating gives
\[
 F'(\alpha)=\log\frac{(1-2\alpha)^2}{\alpha(1-\alpha)}.
\]
Hence $F$ has its unique maximum on $(0,1/2)$ at $\alpha_0$ from
\eqref{eq:alpha0}.  Substituting $\alpha_0$ into \eqref{eq:F-alpha} gives
\[
 \exp(F(\alpha_0))=\varphi.
\]
Therefore
\[
 \max_{0\le k\le \lfloor n/2\rfloor} c_{n,k}
 =\Theta\left(\frac{\varphi^n}{\sqrt n}\right).
\]
Hence we get
\[
 \norm{f_{a^n}}_\infty=\Theta\left(\frac{\varphi^n}{\sqrt n}\right),
\]
as required.
This completes the proof.

\end{proof}

\begin{rem}[Random coefficient growth]\label{rem:random-coeff-growth}
For random cyclically reduced words, computations suggest exponential growth of
$\norm{f_w}_1$ with a base considerably smaller than the crude universal upper bound $\varphi_0$ obtained above.
For two sample words computed using the recursion, the observed values were:
\[
 n=40,
 \qquad
 \norm{f_w}_1=473385,
 \qquad
 \norm{f_w}_1^{1/40}\approx1.386,
\]
and
\[
 n=60,
 \qquad
 \norm{f_w}_1=822733395,
 \qquad
 \norm{f_w}_1^{1/60}\approx1.408.
\]
\end{rem}

\section{Support-size and output-size lower bounds}\label{sec:support-lower}

The degree bound \(\deg f_w\le n=||w||_\mathcal X\) implies the universal support bound
\[
 \#\supp(f_w)\le \binom{n+3}{3}=O(n^3).
\]
Thus, for example, no sequence can have \(\Theta(n^4)\) nonzero monomials.  Nevertheless, the support can grow faster than linearly.  This section first records a general syllable-slice lower bound and then applies it to both an explicit family and to random words.  The same slice also gives cubic lower bounds for total coefficient bit-size.

For \(m\ge1\), put
\[
 E_m(t)=U_{m-1}(t/2),
\]
where \(U_j\) denotes the Chebyshev polynomial of the second kind.  We also use the convention
\(E_0(t)=0\).  Thus
\[
 E_0(t)=0,
 \qquad
 E_1(t)=1,
\]
and for every \(X\in \SL(2,\C)\),
\begin{equation}\label{eq:Xm-Em}
 X^m=E_m(\Tr X)X-E_{m-1}(\Tr X)I.
\end{equation}
The polynomial \(E_m\), of degree $m-1$, has the explicit form
\begin{equation}\label{eq:Em-explicit}
 E_m(t)=\sum_{j=0}^{\lfloor (m-1)/2\rfloor}(-1)^j\binom{m-1-j}{j}t^{m-1-2j}.
\end{equation}
In particular
\[
 \#\supp(E_m)=\left\lceil\frac{m}{2}\right\rceil.
\]

We shall use the following elementary consequence of the explicit formula.  If
$m_1,\dots,m_r\ge 1$ and
\[
 P(t)=\prod_{i=1}^r E_{m_i}(t),
\]
then
\begin{equation}\label{eq:product-Em-support}
 \#\supp(P)=1+\sum_{i=1}^r \left\lfloor\frac{m_i-1}{2}\right\rfloor .
\end{equation}
Indeed, writing
\[
 E_m(t)=t^{m-1}Q_m(t^{-2}),
\]
where
\[
 Q_m(u)=\sum_{j=0}^{\lfloor(m-1)/2\rfloor}(-1)^j
 \binom{m-1-j}{j}u^j,
\]
each $Q_m$ has all degrees from $0$ to $\lfloor(m-1)/2\rfloor$, and its coefficients have the fixed alternating sign pattern $(-1)^j$ with positive absolute values.  Therefore the coefficient of $u^j$ in $\prod_i Q_{m_i}(u)$ has sign $(-1)^j$ and positive absolute value whenever $0\le j\le \sum_i\lfloor(m_i-1)/2\rfloor$.  This proves \eqref{eq:product-Em-support}.

\begin{thm}[A syllable support lower bound]\label{thm:general-syllable-support-lower}
Let
\[
 w=a^{\alpha_1}b^{\beta_1}\cdots a^{\alpha_s}b^{\beta_s}
\]
be a cyclically reduced word in standard syllable form, where $s\ge 1$ and all
$\alpha_i,\beta_i$ are nonzero.  Put
\[
 \kappa_a(w)=\sum_{i=1}^s \left\lfloor\frac{|\alpha_i|-1}{2}\right\rfloor,
 \qquad
 \kappa_b(w)=\sum_{i=1}^s \left\lfloor\frac{|\beta_i|-1}{2}\right\rfloor .
\]
Then
\[
 \#\supp(f_w)\ge (\kappa_a(w)+1)(\kappa_b(w)+1).
\]
More precisely, if
\[
 f_w(x,y,z)=\sum_{k=0}^s g_k(x,y)z^k,
\]
then
\begin{equation}\label{eq:top-z-slice-syllables}
 g_s(x,y)=\pm \prod_{i=1}^s E_{|\alpha_i|}(x)
            \prod_{i=1}^s E_{|\beta_i|}(y),
\end{equation}
and hence
\[
 \#\supp(g_s)=(\kappa_a(w)+1)(\kappa_b(w)+1).
\]
\end{thm}

\begin{proof}
The assertion about the top $z$-coefficient is essentially the final part of the
Horowitz triangular specialization in \cref{lem:horowitz-specialization-precise}.  Write
\[
 f_w=\sum_{k=0}^s g_k(x,y)z^k.
\]
By \cref{lem:horowitz-specialization-precise}, after the substitution
$x=\lambda+\lambda^{-1}$ and $y=\mu+\mu^{-1}$, the coefficient $g_s$ satisfies
\[
 g_s(\lambda+\lambda^{-1},\mu+\mu^{-1})=
 \prod_{i=1}^s
 \frac{\lambda^{\alpha_i}-\lambda^{-\alpha_i}}{\lambda-\lambda^{-1}}
 \prod_{i=1}^s
 \frac{\mu^{\beta_i}-\mu^{-\beta_i}}{\mu-\mu^{-1}}.
\]
For $m>0$,
\[
 \frac{\lambda^m-\lambda^{-m}}{\lambda-\lambda^{-1}}
 =E_m(\lambda+\lambda^{-1}),
\]
and for $m<0$ the same quotient is
$-E_{|m|}(\lambda+\lambda^{-1})$.  Hence the right hand side is
\[
 \pm \prod_{i=1}^s E_{|\alpha_i|}(\lambda+\lambda^{-1})
      \prod_{i=1}^s E_{|\beta_i|}(\mu+\mu^{-1}).
\]
The substitution homomorphism
\[
 \Z[x,y]\to \Z[\lambda^{\pm1},\mu^{\pm1}],
 \qquad x\mapsto \lambda+\lambda^{-1},\quad y\mapsto \mu+\mu^{-1},
\]
is injective, as noted in the proof of \cref{lem:horowitz-specialization-precise}.  This proves
\eqref{eq:top-z-slice-syllables}.

It remains only to count the monomials in this coefficient slice.  By
\eqref{eq:product-Em-support},
\[
 \#\supp\left(\prod_{i=1}^s E_{|\alpha_i|}(x)\right)=\kappa_a(w)+1
\]
and
\[
 \#\supp\left(\prod_{i=1}^s E_{|\beta_i|}(y)\right)=\kappa_b(w)+1.
\]
The variables $x$ and $y$ are distinct, so multiplying the two one-variable products gives exactly
$(\kappa_a(w)+1)(\kappa_b(w)+1)$ distinct monomials in $g_s(x,y)$.  Since $g_s(x,y)z^s$ is a coefficient slice of $f_w$, all these monomials occur in $f_w$ as well.  Therefore
\[
 \#\supp(f_w)\ge \#\supp(g_s)=(\kappa_a(w)+1)(\kappa_b(w)+1).
\]

\end{proof}

\begin{thm}[A cubic coefficient from long syllables]\label{thm:syllable-cubic-coefficient}
Let
\[
 w=a^{\alpha_1}b^{\beta_1}\cdots a^{\alpha_s}b^{\beta_s}
\]
be a cyclically reduced word in standard syllable form, where $s\ge 1$ and all
$\alpha_i,\beta_i$ are nonzero.  Put
\[
 D_a(w)=\sum_{i=1}^s (|\alpha_i|-1),
 \qquad
 D_b(w)=\sum_{i=1}^s (|\beta_i|-1),
\]
and
\[
 L_a(w)=\#\{i: |\alpha_i|\ge 3\}.
\]
There is an absolute constant $M>0$ such that, whenever $L_a(w)\ge 3$, the coefficient
\[
 [x^{D_a(w)-6}y^{D_b(w)}z^s]f_w
\]
satisfies
\[
 \binom{L_a(w)}{3}
 \le
 \left|[x^{D_a(w)-6}y^{D_b(w)}z^s]f_w\right|
 \le M |w|^3 .
\]
In particular, if $L_a(w)\ge \delta |w|$ for some $\delta>0$, then $f_w$ has a coefficient whose absolute value is bounded above and below by positive constant multiples of $|w|^3$, where the constants depend only on $\delta$.
\end{thm}

\begin{proof}
Write
\[
 f_w(x,y,z)=\sum_{k=0}^s g_k(x,y)z^k.
\]
By \cref{eq:top-z-slice-syllables},
\[
 g_s(x,y)=\pm
 \prod_{i=1}^s E_{|\alpha_i|}(x)
 \prod_{i=1}^s E_{|\beta_i|}(y).
\]
For $m\ge 1$ write
\[
 E_m(t)=t^{m-1}Q_m(t^{-2}),
 \qquad
 Q_m(u)=\sum_{r=0}^{\lfloor(m-1)/2\rfloor}
 (-1)^r\binom{m-1-r}{r}u^r.
\]
Then
\[
 \prod_{i=1}^s E_{|\alpha_i|}(x)
 =x^{D_a(w)}\prod_{i=1}^s Q_{|\alpha_i|}(x^{-2}),
\]
and
\[
 \prod_{i=1}^s E_{|\beta_i|}(y)
 =y^{D_b(w)}\prod_{i=1}^s Q_{|\beta_i|}(y^{-2}).
\]
The coefficient of $y^{D_b(w)}$ in the second product is $1$.  Therefore, up to an overall sign, the coefficient of $x^{D_a(w)-6}y^{D_b(w)}z^s$ in $f_w$ is the coefficient of $u^3$ in
\[
 \prod_{i=1}^s Q_{|\alpha_i|}(u).
\]
All contributions to the coefficient of $u^3$ have the same sign $(-1)^3$, and hence there is no cancellation.  For each index $i$ with $|\alpha_i|\ge 3$, the coefficient of $u$ in $Q_{|\alpha_i|}(u)$ has absolute value $|\alpha_i|-2\ge 1$.  Choosing the $u$-term from any three such factors and the constant term from all remaining factors gives a contribution of absolute value at least $1$.  Hence
\[
 \left|[x^{D_a(w)-6}y^{D_b(w)}z^s]f_w\right|
 \ge \binom{L_a(w)}{3}.
\]

It remains to give a uniform cubic upper bound for this same coefficient.  Since the signs are aligned by total $u$-degree, its absolute value is the coefficient of $u^3$ in
\[
 \prod_{i=1}^s \widetilde Q_{|\alpha_i|}(u),
 \qquad
 \widetilde Q_m(u)=\sum_{r=0}^{\lfloor(m-1)/2\rfloor}
 \binom{m-1-r}{r}u^r.
\]
Only the coefficients of $u,u^2,u^3$ can enter.  For an absolute constant $M_0$ and all $m\ge 1$, these coefficients are bounded above by
\[
 M_0m,
 \qquad M_0m^2,
 \qquad M_0m^3,
\]
respectively.  Thus the coefficient of $u^3$ in the product is bounded above by a constant times
\[
 \sum_i |\alpha_i|^3
 +\sum_{i\ne j}|\alpha_i|^2|\alpha_j|
 +\sum_{i<j<k}|\alpha_i||\alpha_j||\alpha_k|.
\]
This is $O((\sum_i |\alpha_i|)^3)$, and $\sum_i |\alpha_i|\le |w|$.  Therefore the coefficient is at most $M|w|^3$ for some absolute constant $M$.

If $L_a(w)\ge \delta |w|$, then, after decreasing the lower constant to absorb the finitely many small values of $|w|$, the lower bound $\binom{L_a(w)}{3}$ is at least a positive constant multiple of $|w|^3$.  The upper bound has just been proved.
\end{proof}

\begin{thm}[A cubic total coefficient bit-size top slice from long syllables]\label{thm:syllable-cubic-bit-size}
Let
\[
 w=a^{\alpha_1}b^{\beta_1}\cdots a^{\alpha_s}b^{\beta_s}
\]
be a cyclically reduced word in standard syllable form, where $s\ge 1$ and all
$\alpha_i,\beta_i$ are nonzero.  Put
\[
 L_a(w)=\#\{i:|\alpha_i|\ge 3\},
 \qquad
 \kappa_b(w)=\sum_{i=1}^s \left\lfloor\frac{|\beta_i|-1}{2}\right\rfloor .
\]
Write
\[
 f_w(x,y,z)=\sum_{k=0}^s g_k(x,y)z^k.
\]
There is an absolute constant $c_0>0$ such that
\[
 \norm{g_s}_{\bit,1}\ge c_0(\kappa_b(w)+1)L_a(w)^2.
\]
Consequently, if $L_a(w)\ge \delta |w|$ and $\kappa_b(w)\ge \delta |w|$ for some $\delta>0$, then
\[
 \norm{f_w}_{\bit,1}\ge \norm{g_s}_{\bit,1}\ge c_\delta |w|^3
\]
for some $c_\delta>0$ depending only on $\delta$.
\end{thm}

\begin{proof}
By \eqref{eq:top-z-slice-syllables},
\[
 g_s(x,y)=\pm P_a(x)P_b(y),
 \qquad
 P_a(x)=\prod_{i=1}^sE_{|\alpha_i|}(x),\quad
 P_b(y)=\prod_{i=1}^sE_{|\beta_i|}(y).
\]
Write
\[
 E_m(t)=t^{m-1}Q_m(t^{-2}),
 \qquad
 Q_m(u)=\sum_{r=0}^{\lfloor(m-1)/2\rfloor}
 (-1)^r\binom{m-1-r}{r}u^r.
\]
All contributions of a fixed $u$-degree in the product $\prod_iQ_{|\alpha_i|}(u)$ have the same sign.  Therefore absolute values of its coefficients may be computed without cancellation.  For every factor with $|\alpha_i|\ge3$, the constant coefficient of $Q_{|\alpha_i|}$ is $1$ and the absolute value of the coefficient of $u$ is $|\alpha_i|-2\ge1$.  Hence the absolute coefficient sequence of $\prod_iQ_{|\alpha_i|}(u)$ dominates the coefficient sequence of $(1+u)^{L_a(w)}$ coefficientwise.  It follows that
\[
 \norm{P_a}_{\bit,1}
 \ge \sum_{r=0}^{L_a(w)}\bit\binom{L_a(w)}{r}.
\]
The last sum is bounded below by a positive constant multiple of $L_a(w)^2$: for example, for linearly many $r$ with $L_a(w)/3\le r\le 2L_a(w)/3$, the binomial coefficient $\binom{L_a(w)}r$ is at least $2^{cL_a(w)}$ for an absolute constant $c>0$, after decreasing $c$ to handle small values of $L_a(w)$.

By \eqref{eq:product-Em-support}, the polynomial $P_b(y)$ has exactly $\kappa_b(w)+1$ nonzero coefficients, and each of them is a nonzero integer.  Since $x$ and $y$ are distinct variables, the nonzero coefficients of $P_a(x)P_b(y)$ are the products of a nonzero coefficient of $P_a$ and a nonzero coefficient of $P_b$, with no collisions.  Multiplication by a nonzero integer cannot decrease binary length.  Thus
\[
 \norm{g_s}_{\bit,1}=\norm{P_aP_b}_{\bit,1}
 \ge (\kappa_b(w)+1)\norm{P_a}_{\bit,1}
 \ge c_0(\kappa_b(w)+1)L_a(w)^2.
\]
The final assertion follows because $g_s(x,y)z^s$ is a coefficient slice of $f_w$ and hence contributes disjoint monomials to $f_w$.
\end{proof}

The next theorem extracts a particularly simple deterministic family from the preceding syllable-slice estimates.

\begin{thm}[A quadratic support lower bound]\label{thm:quadratic-support-lower}
For $m\ge 1$ let 
\[
 w_m=a^m b^m (ab)^m.
\]
Then \(||w_m||_{\mathcal X}=4m\), and
\[
 \#\supp(f_{w_m})\ge \left\lceil\frac{m}{2}\right\rceil^2.
\]
In particular, the worst-case support size of trace polynomials of words of length \(n\) is at least quadratic in \(n\).
\end{thm}

\begin{proof}
We first explain precisely what is being computed.  
Let $A, B\in \SL(2,\C)$ be two arbitrary matrices, and  put
\[
 x=\Tr(A),\qquad y=\Tr(B),\qquad z=\Tr(AB),
\]
and set
\[
 C=AB.
\]
By definition of the trace polynomial, for every such pair $A,B$ one has
\[
 f_{w_m}(x,y,z)=\Tr(A^mB^mC^m),
\]
where $x,y,z$ are the corresponding trace coordinates.  The calculation below is a symbolic calculation of this universal trace polynomial.  Since the equality holds for arbitrary $A,B$, we interpret the final expression as an identity in $\Z[x,y,z]$.  

We use the Cayley--Hamilton formula in the form \eqref{eq:Xm-Em}.  Applied to the three matrices $A$, $B$, and $C=AB$, it gives
\[
 A^m=E_m(x)A-E_{m-1}(x)I,
\]
\[
 B^m=E_m(y)B-E_{m-1}(y)I,
\]
and
\[
 C^m=E_m(z)C-E_{m-1}(z)I.
\]
Substituting these three identities into $\Tr(A^mB^mC^m)$ means that we multiply, in the matrix algebra, the three two-term factors
\[
 \bigl(E_m(x)A-E_{m-1}(x)I\bigr)
 \bigl(E_m(y)B-E_{m-1}(y)I\bigr)
 \bigl(E_m(z)C-E_{m-1}(z)I\bigr),
\]
and then take the trace of the resulting matrix,

\begin{equation}\label{eq:tr-wm}
f_{w_m}=\Tr A^m B^m C^m= \Tr  \, \bigl(E_m(x)A-E_{m-1}(x)I\bigr)
 \bigl(E_m(y)B-E_{m-1}(y)I\bigr)
 \bigl(E_m(z)C-E_{m-1}(z)I\bigr)
\end{equation}

Expanding this product gives eight summands.  In each summand, the scalar polynomial factors such as $E_m(x)$ and $E_{m-1}(y)$ multiply the trace of a shorter matrix word in $A,B,C$.
Recall that $\Tr A=x$, $\Tr(B)=y$ and $\Tr C=z$.

The summand in expanding \eqref{eq:tr-wm} obtained by choosing the first term from each of the three factors is
\[
 E_m(x)E_m(y)E_m(z)\Tr(ABC).
\]
Since $C=AB$, we have
\[
 ABC=ABAB=C^2.
\]
Hence
\[
 \Tr(ABC)=\Tr(C^2)=z^2-2.
\]
Therefore this contribution is
\[
 E_m(x)E_m(y)E_m(z)(z^2-2).
\]
The polynomial $E_m(z)$ has degree $m-1$ and leading coefficient $1$.  Therefore the coefficient of $z^{m+1}$ in this contribution is exactly
\[
 E_m(x)E_m(y).
\]

We now check that none of the other seven summands in \eqref{eq:tr-wm} contributes to the coefficient of $z^{m+1}$.  First suppose that the $C^m$ factor contributes the second term $-E_{m-1}(z)I$.  Then the explicit polynomial factor in $z$ has degree at most $m-2$.  The remaining trace is one of
\[
 \Tr(AB)=z,\qquad \Tr(A)=x,\qquad \Tr(B)=y,\qquad \Tr(I)=2,
\]
so it has $z$-degree at most $1$.  Hence every term using $E_{m-1}(z)$ has total $z$-degree at most $m-1$.

It remains to consider terms using the first $C$-summand $E_m(z)C$ but at least one identity summand from the $A^m$ or $B^m$ factors.  The relevant traces are
\[
 \Tr(AC)=\Tr(A^2B)=xz-y,
\]
\[
 \Tr(BC)=\Tr(BAB)=yz-x,
\]
and
\[
 \Tr(C)=z.
\]
Each has $z$-degree at most $1$.  Since $E_m(z)$ has degree $m-1$, these terms have total $z$-degree at most $m$.  Thus they also cannot contribute to the coefficient of $z^{m+1}$.

Combining the above computations for \eqref{eq:tr-wm}, we obtain the exact coefficient-slice identity
\[
 [z^{m+1}]f_{w_m}(x,y,z)=E_m(x)E_m(y),
\]
where $[z^{m+1}]$ means the coefficient of $z^{m+1}$ after viewing $f_{w_m}$ as a polynomial in $z$ with coefficients in $\Z[x,y]$.

The supports of $E_m(x)$ and $E_m(y)$ have size $\lceil m/2\rceil$.  Therefore every product of a monomial occurring in $E_m(x)$ with a monomial occurring in $E_m(y)$ gives a distinct nonzero monomial in the coefficient slice above.  Equivalently, all monomials
\[
 x^i y^j z^{m+1},
\]
with $x^i$ occurring in $E_m(x)$ and $y^j$ occurring in $E_m(y)$, occur nontrivially in $f_{w_m}$.  Hence
\[
 \#\supp(f_{w_m})\ge
 \#\supp(E_m(x))\#\supp(E_m(y))
 =\lceil m/2\rceil^2.
\]
Since $||w_m||_{\mathcal X}=4m$, the worst-case support is at least quadratic in the word length.

\end{proof}

\begin{lem}[A one-sided Chernoff bound for adapted indicators]\label{lem:adapted-chernoff}
Let $I_1,\dots,I_N$ be $0$--$1$ random variables adapted to a filtration
$\mathcal F_0\subseteq \mathcal F_1\subseteq \cdots\subseteq \mathcal F_N$, with
$I_j$ measurable with respect to $\mathcal F_j$.  Suppose that for some
$0<\rho\le 1$ one has
\[
 \mathbb E(I_j\mid \mathcal F_{j-1})\ge \rho
\]
for every $j$.  Then
\[
 \Pr\left(\sum_{j=1}^N I_j\le \frac{\rho N}{2}\right)\le \exp(-\rho N/8).
\]
\end{lem}

\begin{proof}
This is the usual multiplicative Chernoff lower-tail estimate for sub-Bernoulli sums; for completeness we recall a short proof.  For $\theta>0$ and $0\le I\le1$,
\[
 e^{-\theta I}\le 1-(1-e^{-\theta})I.
\]
Thus, conditioning successively and using
$\mathbb E(I_j\mid \mathcal F_{j-1})\ge\rho$, we get
\[
 \mathbb E e^{-\theta\sum I_j}
 \le \bigl(1-\rho(1-e^{-\theta})\bigr)^N
 \le \exp\bigl(-\rho N(1-e^{-\theta})\bigr).
\]
Therefore
\[
 \Pr\left(\sum I_j\le \rho N/2\right)
 \le \exp\bigl(\theta\rho N/2-\rho N(1-e^{-\theta})\bigr).
\]
Taking $\theta=\log 2$ gives an exponent at most $-\rho N/8$.
\end{proof}

\begin{thm}[Quadratic support, a cubic coefficient, and cubic total bit-size for random positive words]\label{thm:random-positive-support}
Let $W_n$ be a random positive word of length $n$ in the alphabet $\{a,b\}$, where the letters are independent and $\Pr(a)=p\in(0,1)$.  Then there are constants $c_p,c'_p,c''_p,c'''_p,C_p,\lambda_p>0$ such that the following three inequalities hold for all $n\ge1$:
\begin{enumerate}
\item
\[
 \Pr\bigl(\#\supp(f_{W_n})\ge c_p n^2\bigr)\ge 1-C_p e^{-\lambda_p n}.
\]
\item
\[
 \Pr\left(\exists i,j,k\text{ with } c'_p n^3\le
 |[x^iy^jz^k]f_{W_n}|\le c''_p n^3\right)\ge 1-C_p e^{-\lambda_p n}.
\]
\item
\[
 \Pr\bigl(\norm{f_{W_n}}_{\bit,1}\ge c'''_p n^3\bigr)\ge 1-C_p e^{-\lambda_p n}.
\]
\end{enumerate}
\end{thm}

\begin{proof}
Partition the first $5\lfloor n/5\rfloor$ letters of $W_n$ into disjoint consecutive blocks of length $5$.  Let $N=\lfloor n/5\rfloor$.  Call a block $a$-good if it is exactly
\[
 baaab,
\]
and call it $b$-good if it is exactly
\[
 abbba.
\]
The numbers $X_a$ and $X_b$ of $a$-good and $b$-good blocks are binomial random variables with parameters
\[
 \Pr(baaab)=(1-p)^2p^3,
 \qquad
 \Pr(abbba)=p^2(1-p)^3,
\]
respectively.  Hence, by the ordinary Chernoff bound, there are constants
$c_1,C_1,\lambda_1>0$, depending only on $p$, such that with probability at least
$1-C_1e^{-\lambda_1 n}$ one has
\begin{equation}\label{eq:positive-good-blocks-linear}
 X_a\ge c_1 n,
 \qquad
 X_b\ge c_1 n.
\end{equation}

Assume \eqref{eq:positive-good-blocks-linear}.  Each $a$-good block contains the substring
$baaab$ and therefore determines a cyclic $a$-syllable of length at least $3$ in the cyclic word represented by $W_n$.  Distinct disjoint good blocks determine distinct such syllables, because the three displayed $a$'s are bounded on both sides by $b$'s.  Hence, for the standard cyclic syllable form of $W_n$,
\[
 \sum_i \left\lfloor\frac{\alpha_i-1}{2}\right\rfloor \ge X_a\ge c_1 n.
\]
Similarly,
\[
 \sum_i \left\lfloor\frac{\beta_i-1}{2}\right\rfloor \ge X_b\ge c_1 n.
\]
In particular $W_n$ involves both generators.  Applying \cref{thm:general-syllable-support-lower} gives
\[
 \#\supp(f_{W_n})\ge (c_1n+1)^2\ge c_p n^2
\]
for a suitable $c_p>0$, after decreasing $c_p$ to handle the finitely many small values of $n$ and increasing $C_p$ if necessary.  This proves part~(1).

The same event also gives at least $c_1n$ cyclic $a$-syllables of length at least $3$.  Therefore, in the notation of \cref{thm:syllable-cubic-coefficient}, $L_a(W_n)\ge c_1n$.  That theorem gives a specific coefficient of $f_{W_n}$ whose absolute value is at least a positive constant multiple of $n^3$ and at most a positive constant multiple of $n^3$.  After adjusting constants for the finitely many small values of $n$, this proves part~(2).

Moreover, on the same event we have $L_a(W_n)\ge c_1n$ and $\kappa_b(W_n)\ge c_1n$.  Applying \cref{thm:syllable-cubic-bit-size} gives
\[
 \norm{f_{W_n}}_{\bit,1}\ge c'''_p n^3
\]
for a suitable constant $c'''_p>0$, again after adjusting constants for the finitely many small values of $n$.  This proves part~(3).
\end{proof}

\begin{thm}[Quadratic support, a cubic coefficient, and cubic total bit-size for random freely reduced words]\label{thm:random-freely-reduced-support}
Let $W_n$ be a random freely reduced word of length $n$ in $F(a,b)$, generated by the standard nonbacktracking model.  Then there are constants $c,c',c'',c''',C,\lambda>0$ such that the following three inequalities hold for all $n\ge1$:
\begin{enumerate}
\item
\[
 \Pr\bigl(\#\supp(f_{W_n})\ge c n^2\bigr)\ge 1-Ce^{-\lambda n}.
\]
\item
\[
 \Pr\left(\exists i,j,k\text{ with } c'n^3\le
 |[x^iy^jz^k]f_{W_n}|\le c''n^3\right)\ge 1-Ce^{-\lambda n}.
\]
\item
\[
 \Pr\bigl(\norm{f_{W_n}}_{\bit,1}\ge c''' n^3\bigr)\ge 1-Ce^{-\lambda n}.
\]
\end{enumerate}
\end{thm}

\begin{proof}
Again partition the first $5N$ letters into disjoint consecutive blocks of length $5$, where $N=\lfloor n/5\rfloor$.  Call a block $a$-good if its first and last letters are in
$\{b,b^{-1}\}$ and its three middle letters are all equal to $a$ or are all equal to $a^{-1}$.  Thus an $a$-good block contains one of the patterns
\[
 b^{\epsilon} a^{\delta}a^{\delta}a^{\delta} b^{\eta},
 \qquad \epsilon,\delta,\eta\in\{\pm1\}.
\]
Define $b$-good blocks symmetrically, with the roles of $a$ and $b$ interchanged.

Let $I_j^a$ be the indicator that the $j$-th block is $a$-good.  Conditional on all letters before that block, the probability that the block is $a$-good is bounded below by an absolute positive constant.  Indeed, the first letter of the block can be chosen to be one of $b,b^{-1}$ in at least one admissible way, the next letter can be chosen to be one of $a,a^{-1}$ in two admissible ways, the next two letters are then forced to repeat that same $a$-type letter, and the final letter can be chosen to be one of $b,b^{-1}$ in two admissible ways.  Since each non-initial transition has three choices, we may take for instance
\[
 \rho=\frac{1}{3}\cdot\frac{2}{3}\cdot\frac{1}{3}\cdot\frac{1}{3}\cdot\frac{2}{3}=\frac{4}{243}.
\]
For the first block the same lower bound is valid, since the first letter is uniform and $\Pr(s_1\in\{b,b^{-1}\})=1/2\ge1/3$.  Hence \cref{lem:adapted-chernoff} gives
\[
 \Pr\left(\sum_{j=1}^N I_j^a\le \rho N/2\right)\le e^{-\rho N/8}.
\]
The same argument applies to the number of $b$-good blocks.  Therefore, with probability at least $1-C_1e^{-\lambda_1 n}$, the word $W_n$ contains at least $\rho N/2$ $a$-good blocks and at least $\rho N/2$ $b$-good blocks.

Let $K_n$ be the amount of cyclic cancellation, as in \cref{lem:cyclic-cancellation-random}.  Choose a small absolute constant $\eta>0$ such that, for all sufficiently large $n$,
\[
 \frac{\rho}{2}\left\lfloor\frac n5\right\rfloor-2\eta n-2\ge c_1 n
\]
for some $c_1>0$.  By \cref{lem:cyclic-cancellation-random},
\[
 \Pr(K_n\ge \eta n)\le 3^{1-\eta n}.
\]
When $K_n<\eta n$, passing from $W_n$ to its cyclically reduced core removes fewer than $2\eta n$ letters, all from the two ends of the linear word.  Call a good block surviving if none of its five letters is removed in this cyclic cancellation.  Since the good blocks are disjoint consecutive blocks, this can destroy at most $2\eta n+2$ of them.  Thus, with probability at least $1-C_2e^{-\lambda_2n}$, the cyclically reduced core $W_n^{\cyc}$ contains at least $c_1n$ surviving $a$-good blocks and at least $c_1n$ surviving $b$-good blocks.

Each surviving $a$-good block gives an $a$-syllable in $W_n^{\cyc}$ of absolute length at least $3$, and distinct surviving good blocks give distinct such syllables, because the middle three $a$-type letters are bounded on both sides by $b$-type letters inside the block.  Therefore, if
\[
 W_n^{\cyc}=a^{\alpha_1}b^{\beta_1}\cdots a^{\alpha_s}b^{\beta_s}
\]
is written in standard cyclic syllable form, then
\[
 \sum_i \left\lfloor\frac{|\alpha_i|-1}{2}\right\rfloor \ge c_1n,
 \qquad
 \sum_i \left\lfloor\frac{|\beta_i|-1}{2}\right\rfloor \ge c_1n,
\]
and also $L_a(W_n^{\cyc})\ge c_1n$ in the notation of \cref{thm:syllable-cubic-coefficient}.  The trace polynomial is invariant under conjugacy, so $f_{W_n}=f_{W_n^{\cyc}}$.

Applying \cref{thm:general-syllable-support-lower} to $W_n^{\cyc}$ gives
\[
 \#\supp(f_{W_n})=\#\supp(f_{W_n^{\cyc}})\ge (c_1n+1)^2\ge cn^2
\]
for a suitable absolute constant $c>0$, after adjusting constants for the finitely many small values of $n$.  This proves part~(1).  Applying \cref{thm:syllable-cubic-coefficient} to $W_n^{\cyc}$ gives a specific coefficient of $f_{W_n}$ whose absolute value is at least a positive constant multiple of $n^3$ and at most a positive constant multiple of $n^3$.  After adjusting constants for small $n$, this proves part~(2).

On the same event we have $L_a(W_n^{\cyc})\ge c_1n$ and $\kappa_b(W_n^{\cyc})\ge c_1n$.  Applying \cref{thm:syllable-cubic-bit-size} to $W_n^{\cyc}$ and using $f_{W_n}=f_{W_n^{\cyc}}$ gives
\[
 \norm{f_{W_n}}_{\bit,1}\ge c''' n^3
\]
for a suitable absolute constant $c'''>0$, after adjusting constants for small $n$.  This proves part~(3).
\end{proof}

\begin{cor}[Random quadratic support, cubic coefficient size, cubic total bit-size, and sparse-output lower bounds]\label{cor:random-expanded-output-lower}
The following hold.
\begin{enumerate}
\item Let $W_n$ be either a random positive word of length $n$ in the alphabet $\{a,b\}$, with independent letters satisfying $\Pr(a)=p\in(0,1)$, or a random freely reduced word of length $n$ in $F(a,b)$ generated by the standard nonbacktracking model.  Then there are constants $c,C,\lambda>0$, depending on the model and on $p$ in the positive case, such that
\[
 \Pr\bigl(\#\supp(f_{W_n})\ge c n^2\bigr)
 \ge 1-Ce^{-\lambda n}.
\]
\item For the same two random models, after adjusting the constants $c,C,\lambda>0$ if necessary,
\[
 \Pr\left(\exists i,j,k\text{ with } c n^3\le
 |[x^iy^jz^k]f_{W_n}|\le c^{-1}n^3\right)
 \ge 1-Ce^{-\lambda n}.
\]
\item For the same two random models, after adjusting the constants $c,C,\lambda>0$ if necessary,
\[
 \Pr\bigl(\norm{f_{W_n}}_{\bit,1}\ge c n^3\bigr)
 \ge 1-Ce^{-\lambda n}.
\]
Thus the total binary bit-size of the nonzero coefficients in the fully expanded sparse output grows at least cubically with exponentially high probability.
\item Fix one of the two random models in part~$\mathrm{(1)}$.  There are constants $c_4,C_4,\lambda_4>0$, depending on the model, on $p$ in the positive case, and on the precise sparse-output convention, such that for every deterministic algorithm $\mathcal A$ which computes and outputs $f_w$ in fully expanded sparse binary form on input words $w$ of length $n$ in the same model, if $T_{\mathcal A}(W_n)$ denotes the running time of $\mathcal A$ on input $W_n$, then
\[
 \Pr\bigl(T_{\mathcal A}(W_n)\ge c_4 n^3\bigr)\ge 1-C_4e^{-\lambda_4 n}
\]
for all $n\ge1$.
\item There are constants $c_5,C_5,\lambda_5>0$, depending on the precise sparse-output convention, such that for every deterministic algorithm $\mathcal A$ which computes and outputs $f_w$ in fully expanded sparse binary form on input cyclically reduced words $w$ of length $n$, if $U_n$ is uniformly random among cyclically reduced words of length $n$ and $T_{\mathcal A}(U_n)$ denotes the running time of $\mathcal A$ on input $U_n$, then
\[
 \Pr\bigl(T_{\mathcal A}(U_n)\ge c_5 n^3\bigr)\ge 1-C_5e^{-\lambda_5 n}
\]
for all $n\ge1$.
\end{enumerate}
Equivalently, any deterministic fully expanded sparse-output algorithm has running time $\Omega(n^3)$ with exponentially high probability in the positive, freely reduced, and uniformly cyclically reduced models, with the preceding qualification about positive words applying only to the positive model.
\end{cor}

\begin{proof}
Parts~$\mathrm{(1)}$--$\mathrm{(3)}$ follow directly from \cref{thm:random-positive-support,thm:random-freely-reduced-support}, by taking the intersection of the exponentially high probability events appearing there and by decreasing $c$ and increasing $C$ if necessary.  For part~$\mathrm{(4)}$, fix one of the two random models and choose constants $c,C,\lambda>0$ as in part~$\mathrm{(3)}$.  In fully expanded sparse binary form, all nonzero monomials are explicitly listed and all nonzero coefficients must be written in binary, together with enough monomial data to identify their exponents.  Hence the running time of any deterministic algorithm that writes such an output is bounded below, up to a positive constant depending only on the output convention and machine model, by $\norm{f_w}_{\bit,1}$.  Thus, for a suitable $c_4>0$,
\[
 \Pr\bigl(T_{\mathcal A}(W_n)\ge c_4n^3\bigr)
 \ge \Pr\bigl(\norm{f_{W_n}}_{\bit,1}\ge cn^3\bigr)
 \ge 1-Ce^{-\lambda n}.
\]
Taking $C_4=C$ and $\lambda_4=\lambda$ gives part~$\mathrm{(4)}$.

For part~$\mathrm{(5)}$, use the freely reduced nonbacktracking constants $c,C,\lambda>0$ from part~$\mathrm{(3)}$ and let $\mathcal C_n$ be the set of cyclically reduced words of length $n$.  The same output-size lower bound as above gives a constant $\alpha>0$, depending only on the sparse-output convention and machine model, such that the running time is at least $\alpha\norm{f_w}_{\bit,1}$ on every cyclically reduced input $w$.  Put $c_5=\alpha c$.  If $W_n$ is uniformly distributed on all freely reduced words of length $n$, then conditioning on $W_n\in\mathcal C_n$ gives the uniform distribution on $\mathcal C_n$.  By Rivin's count \cite[Theorem~1.1]{Rivin2010},
\[
 |\mathcal C_n|=3^n+2+(-1)^n,
 \qquad
 |\mathcal R_n|=4\cdot 3^{n-1},
\]
where $\mathcal R_n$ is the set of freely reduced words of length $n$.  Hence
\[
 \Pr(W_n\in\mathcal C_n)=\frac34+O(3^{-n})
\]
and this probability is bounded below by a positive absolute constant.  Therefore
\[
 \Pr\bigl(T_{\mathcal A}(U_n)<c_5n^3\bigr)
 \le \Pr\bigl(\norm{f_{U_n}}_{\bit,1}<cn^3\bigr)
 =\Pr\bigl(\norm{f_{W_n}}_{\bit,1}<cn^3\mid W_n\in\mathcal C_n\bigr)
 \le \frac{\Pr\bigl(\norm{f_{W_n}}_{\bit,1}<cn^3\bigr)}{\Pr(W_n\in\mathcal C_n)}
 \le C_5e^{-\lambda_5 n}
\]
for suitable constants $C_5,\lambda_5>0$.  This proves part~$\mathrm{(5)}$.
\end{proof}

\begin{cor}[A cubic sparse-output lower bound]\label{cor:cubic-output-lower}
There is a sequence of cyclically reduced words \(w_m\) with \(||w_m||_{\mathcal X}=4m\) such that \(\norm{f_{w_m}}_{\bit,1}=\Omega(m^3)\).  Consequently, any deterministic algorithm which outputs the fully expanded polynomial in sparse binary form has worst-case running time \(\Omega(n^3)\) on inputs of length \(n\).
\end{cor}

\begin{proof}
We use the words \(w_m=a^mb^m(ab)^m\) from \cref{thm:quadratic-support-lower}, with $n=4m=|w_m|=|w_m|_\mathcal X=||w_m||_\mathcal X$.  It is enough to look at the coefficient slice of \(z^{m+1}\), namely \(E_m(x)E_m(y)\).

Let
\[
 E_m(t)=\sum_j e_j t^{m-1-2j}.
\]
We first note that the sum of the binary lengths of the nonzero coefficients \(e_j\) is \(\Omega(m^2)\).  Indeed, for all integers \(j\) in the interval
\[
 \left\lceil \frac{m}{5}\right\rceil \le j\le \left\lfloor\frac{m}{4}\right\rfloor
\]
and for all sufficiently large \(m\), the coefficient
\[
 |e_j|=\binom{m-1-j}{j}
\]
is at least \(2^{cm}\) for some absolute constant \(c>0\).  For instance, after replacing the displayed interval by a slightly smaller subinterval if necessary, the elementary bound
\[
 \binom{N}{j}\ge \left(\frac{N-j+1}{j}\right)^j
\]
with \(N=m-1-j\) gives a lower bound \(2^{cm}\).  There are \(\Omega(m)\) such values of \(j\), so
\[
 \sum_{e_j\ne0} \log_2(|e_j|+1)=\Omega(m^2).
\]

The nonzero coefficients of \(E_m(x)E_m(y)\) are the products \(e_i e_j\), with no collisions because the monomials have distinct \((x,y)\)-exponents.  Hence the total binary size of this coefficient slice is bounded below by
\[
 \sum_{i,j}\log_2(|e_i e_j|+1)
 \ge
 \sum_{i,j}\bigl(\log_2(|e_i|+1)+\log_2(|e_j|+1)-1\bigr).
\]
Since \(E_m\) has \(\Omega(m)\) nonzero coefficients and \(\sum_j\log_2(|e_j|+1)=\Omega(m^2)\), the right-hand side is \(\Omega(m^3)\).  Therefore \(\norm{f_{w_m}}_{\bit,1}=\Omega(m^3)\).  Since \(||w_m||_{\mathcal X}=4m\), this is \(\Omega(n^3)\) in terms of input length \(n\).  Any algorithm that actually writes this output must use at least that much time in the worst case.
\end{proof}

\begin{rem}
The preceding deterministic lower bound is complementary to the random lower bounds above: it gives an explicit sequence, while the random results give exponentially generic behavior.  Neither lower bound matches the universal upper bound \(O(n^4)\) on dense expanded output size.  It remains open whether the true worst-case sparse output size is cubic, quartic, or somewhere in between, and whether the support size itself can be cubic in the word length.  Computational evidence suggests that families such as powers of the commutator may have cubic support. 
\end{rem}

\section{Polynomial-time computation in expanded form}\label{sec:polytime}

The recursive trace identities of \cref{sec:algorithm} are conceptually useful, but a more direct way to obtain a polynomial-time bound for the fully expanded polynomial is to evaluate the word on generic matrices over a fixed quadratic extension of \(\Z[x,y,z]\).  This section proves the elementary complexity statement announced in the introduction.

Let
\[
 R=\Z[x,y,z,\zeta]/(\zeta^2-z\zeta+1).
\]
Note that $\Z[x,y,z]$ is naturally embedded in $R$.

Every element of \(R\) has a unique representative of the form
\[
 P(x,y,z)+Q(x,y,z)\zeta,
\]
with \(P,Q\in\Z[x,y,z]\). 

Define matrices over \(R\) by
\[
 \mathfrak A=\begin{pmatrix}x&-1\\ 1&0\end{pmatrix},
\qquad
\mathfrak  B=\begin{pmatrix}0&\zeta\\ \zeta-z&y\end{pmatrix}.
\]

\begin{lem}\label{lem:det}
We have 
\[
\mathfrak A, \mathfrak B\in SL(2,R)
\]
\end{lem}
\begin{proof}
Indeed,
\[
 \det \mathfrak A=1,
 \qquad
 \det \mathfrak B=-(\zeta)(\zeta-z)=-\zeta^2+z\zeta=1,
\]
where the last equality uses the defining relation \(\zeta^2=z\zeta-1\) in \(R\). 
\end{proof}

\begin{defn}
For a word \(w=s_1\cdots s_n\) in the alphabet \(a^{\pm1},b^{\pm1}\), set
\[
 M(a)=\mathfrak A,\qquad M(a^{-1})=\mathfrak A^{-1}=\begin{pmatrix}0&1\\ -1&x\end{pmatrix},\qquad M(b)=\mathfrak B,\qquad M(b^{-1})=\mathfrak B^{-1}=\begin{pmatrix}y&-\zeta\\ z-\zeta&0\end{pmatrix},
\]
and define
\[
 M(w):=M(s_1)\cdots M(s_n)\in M_2(R).
\]
\end{defn}

\begin{prop}[A quadratic generic matrix model]\label{prop:generic-matrix-model}
For every freely reduced word $w\in F(a,b)$ we have 
\[
 \Tr M(w)=f_w(x,y,z)
\]
as an element of \(R\).  In particular, if the unique representative of \(\Tr M(w)\) is written as
\[
 \Tr M(w)=P_w(x,y,z)+Q_w(x,y,z)\zeta,
\]
then
\[
 Q_w=0\qquad\text{and}\qquad P_w=f_w.
\]
\end{prop}

\begin{proof}
By definition of $\mathfrak A, \mathfrak B$, we have  
\[
 \Tr \mathfrak A=x,
 \qquad
 \Tr \mathfrak B=y.
\]
Also, direct multiplication gives
\[
 \mathfrak A\mathfrak B=\begin{pmatrix}z-\zeta&x\zeta-y\\ 0&\zeta\end{pmatrix},
 \qquad
 \Tr(\mathfrak A\mathfrak B)=z.
\]
By the defining property of the trace polynomial, for every pair \(A,B\in\SL(2,\C)\) one has
\[
 \Tr(w(A,B))=f_w(\Tr A,\Tr B,\Tr(AB)).
\]
Equivalently, after using the universal trace identities to express the left hand side in the entries of two generic determinant-one matrices, this is a polynomial identity with integer coefficients.  Hence it may be specialized to any commutative coefficient ring.  Applying this identity over the ring \(R\) to the above matrices gives
\[
 \Tr M(w)=f_w(\Tr \mathfrak A,\Tr \mathfrak B,\Tr(\mathfrak A\mathfrak B))=f_w(x,y,z).
\]
The right-hand side belongs to the subring \(\Z[x,y,z]\subset R\).  Since every element of \(R\) has a unique expression \(P+Q\zeta\), it follows that \(Q_w=0\) and \(P_w=f_w\).
\end{proof}

Before stating the algorithm, we spell out the representation of polynomials used in the complexity estimate.  A polynomial
\[
 P(x,y,z)=\sum_{i+j+k\le n} c_{ijk}x^iy^jz^k
\]
of total degree at most \(n\) is stored in a \emph{dense coefficient array}: namely, we allocate one integer entry for every triple
\[
 (i,j,k)\in\Z_{\ge0}^3,\qquad i+j+k\le n,
\]
and store the corresponding coefficient \(c_{ijk}\), even when \(c_{ijk}=0\).  Thus each such polynomial is represented by \(\binom{n+3}{3}=O(n^3)\) integer entries.  An element of the quotient ring \(R\) in normal form
\[
 P(x,y,z)+Q(x,y,z)\zeta
\]
is stored as two such dense arrays, one for \(P\) and one for \(Q\).  Addition and subtraction are performed coefficient-by-coefficient, while multiplication by \(x\), \(y\), \(z\), or by \(\zeta\) is implemented by shifting entries and then reducing all occurrences of \(\zeta^2\) via the relation \(\zeta^2=z\zeta-1\).

\begin{thm}[Polynomial-time expanded computation]\label{thm:polytime-expanded}
Let \(w\in F(a,b)\) be given as a word of length \(n\) in the alphabet \(a^{\pm1},b^{\pm1}\).  There is an algorithm which outputs the fully expanded Fricke polynomial
\[
 f_w(x,y,z)\in\Z[x,y,z]
\]
in deterministic Turing-machine time \(O(n^5)\) and space \(O(n^4)\), using dense coefficient arrays for all monomials of total degree at most \(n\).  In this dense coefficient-table representation, rather than a sparse list of nonzero monomials, the final output has \(O(n^4)\) bits.
\end{thm}

\begin{proof}
Use the matrices \(\mathfrak A, \mathfrak B\in\SL(2,R)\) and the letter matrices \(M(a),M(a^{-1}),M(b),M(b^{-1})\) from \cref{prop:generic-matrix-model}.  Thus every letter in \(a^{\pm1},b^{\pm1}\) is represented by a \(2\times2\) matrix whose entries are among
\[
 0,\ \pm1,\ x,\ y,\ \pm\zeta,\ z-\zeta,\ \zeta-z.
\]
Given \(w=s_1\cdots s_n\), we multiply the corresponding matrices
\[
 M(w)=M(s_1)\cdots M(s_n)
\]
\emph{inside} the matrix ring \(M_2(R)\). 

More precisely, we do not manipulate arbitrary polynomials in four independent variables \(x,y,z,\zeta\).  Instead, we perform all computations in  the quotient ring
\[
 R=\Z[x,y,z,\zeta]/(\zeta^2-z\zeta+1).
\]
Thus every entry of every intermediate matrix is stored uniquely as a pair
\[
 (P,Q)\quad\text{representing}\quad P(x,y,z)+Q(x,y,z)\zeta,
\]
where \(P,Q\in\Z[x,y,z]\).  Whenever a product produces a factor \(\zeta^2\), it is immediately reduced using
\[
 \zeta^2=z\zeta-1.
\]


 By \cref{prop:generic-matrix-model}, the trace of the resulting matrix $M(w)=M(s_1)\cdots M(s_n)$ is exactly \(f_w(x,y,z)\).  Therefore the algorithm outputs the polynomial part \(P_w\) in the unique expression
\[
 \Tr M(w)=P_w(x,y,z)+Q_w(x,y,z)\zeta;
\]
by \cref{prop:generic-matrix-model}, \(Q_w=0\) and \(P_w=f_w\). 

It remains to bound the running time for computing $\Tr M(w)$ by performing the requisite matrix multiplications in $SL(2,R)$.  After the first \(k\) letters of $w$ have been processed, every polynomial occurring in the entries of the current matrix has total degree at most \(k\).  Therefore it is enough to store coefficient arrays indexed by monomials \(x^iy^jz^\ell\) with
\[
 i+j+\ell\le n.
\]
There are
\[
 \binom{n+3}{3}=O(n^3)
\]
such monomials.

Multiplication of an element \(P+Q\zeta\in R\) by any of the possible generator-matrix entries is  rather simple.  Multiplication by \(x,y,z\) is just a shift of coefficient arrays, multiplication by \(\pm1\) is a sign change, and multiplication by \(\zeta\) is given by
\[
 (P+Q\zeta)\zeta=-Q+(P+zQ)\zeta.
\]
Similarly,
\[
 (P+Q\zeta)(\zeta-z)=-zP-Q+P\zeta,
\]
and
\[
 (P+Q\zeta)(z-\zeta)=zP+Q-P\zeta.
\]
Thus one multiplication of the current \(2\times2\) matrix by a generator matrix requires only a bounded number of additions and shifts of arrays of size \(O(n^3)\).  Hence each letter costs \(O(n^3)\) integer arithmetic operations, and the whole word costs \(O(n^4)\) such operations.

Finally, the bit lengths of the intermediate coefficients are only linear in \(n\).  Indeed, at each step the \(\ell^1\)-norm of the coefficient array of each entry is multiplied by at most an absolute constant: the update consists of a bounded number of additions, sign changes, and monomial shifts.  Hence all intermediate coefficient absolute values are at most \(C^n\) for some universal constant \(C\), so their binary lengths are \(O(n)\).

We now translate the preceding count into an ordinary deterministic Turing-machine bound.  Each letter update involves \(O(n^3)\) coefficient-array entries, and each entry operation is an addition, subtraction, sign change, or copy involving integers of bit length \(O(n)\).  With schoolbook arithmetic bit operations, such an integer operation costs \(O(n)\) time.  Thus each letter costs \(O(n^4)\) Turing-machine steps, and the whole word costs
\[
  O(n)\cdot O(n^4)=O(n^5)
\]
steps.  The space required to store a constant number of coefficient arrays is
\[
  O(n^3)\text{ coefficients}\times O(n)\text{ bits per coefficient}=O(n^4)
\]
bits.  The same bound also covers the final dense output table: it contains one slot for each
monomial of total degree at most \(n\), hence \(O(n^3)\) integer coefficients, and each
coefficient has \(O(n)\) bits.  This proves the stated deterministic Turing-machine time
and space bounds.
\end{proof}

\begin{defn}[Character equivalence]\label{def:character-equivalence}
For words $v,w\in F(a,b)$ we write
\[
 v\equiv_c w
\]
if $v$ and $w$ are $\SL(2,\C)$-character equivalent, that is, if
\[
 \Tr(v(A,B))=\Tr(w(A,B))
\]
for every pair $A,B\in\SL(2,\C)$.
\end{defn}

\begin{lem}[Character equivalence and Fricke polynomials]\label{lem:character-equivalence-fricke}
For any $v,w\in F(a,b)$,
\[
 v\equiv_c w
 \quad\Longleftrightarrow\quad
 f_v=f_w \quad\text{in }\Z[x,y,z].
\]
\end{lem}

\begin{proof}
If $f_v=f_w$, then for every $A,B\in\SL(2,\C)$ we have
\[
 \Tr(v(A,B))
 =f_v(\Tr A,\Tr B,\Tr(AB))
 =f_w(\Tr A,\Tr B,\Tr(AB))
 =\Tr(w(A,B)),
\]
so $v\equiv_c w$.

Conversely, suppose $v\equiv_c w$.  Then the polynomial
\[
 F(x,y,z):=f_v(x,y,z)-f_w(x,y,z)
\]
vanishes on every trace triple
\[
 (x,y,z)=(\Tr A,\Tr B,\Tr(AB))
\]
with $A,B\in\SL(2,\C)$.  We claim that every point of $\C^3$ occurs as such a trace triple.  Indeed, given arbitrary $x,y,z\in\C$, take
\[
 A=\begin{pmatrix}x&-1\\ 1&0\end{pmatrix}.
\]
Choose $q,r\in\C$ satisfying
\[
 q-r=z,
 \qquad
 qr=-1.
\]
This is possible since $r$ may be taken to be any root of $r^2+zr+1=0$ and then $q=r+z$.  Put
\[
 B=\begin{pmatrix}0&q\\ r&y\end{pmatrix}.
\]
Then $\det A=\det B=1$, $\Tr A=x$, $\Tr B=y$, and direct multiplication gives
\[
 \Tr(AB)=q-r=z.
\]
Thus $F$ vanishes on all of $\C^3$.  Since $F\in\Z[x,y,z]\subset\C[x,y,z]$, it follows that $F=0$ as a polynomial.  Hence $f_v=f_w$.
\end{proof}

\begin{cor}[Polynomial-time decision of $\SL(2,\C)$-character equivalence]\label{cor:polytime-character-equivalence}
There is an algorithm which, given two freely reduced words $v,w\in F(a,b)$ of length at most $n$, decides whether $v\equiv_c w$ in deterministic Turing-machine time $O(n^5)$.
\end{cor}

\begin{proof}
Apply the algorithm of \cref{thm:polytime-expanded} to $v$ and to $w$.  This computes the fully expanded coefficient arrays of $f_v$ and $f_w$ in time $O(n^5)$ and space $O(n^4)$ for each word.  By \cref{lem:character-equivalence-fricke}, the words are character equivalent if and only if these two polynomials are equal.

To compare the two expanded polynomials, compare their dense coefficient arrays for all monomials $x^iy^jz^k$ with $i+j+k\le n$.  There are $O(n^3)$ such monomials, and by the coefficient-growth estimates used in \cref{thm:polytime-expanded}, all coefficients produced by the algorithm have $O(n)$ bits.  Hence the coefficient-by-coefficient comparison costs $O(n^4)$ Turing-machine time, which is dominated by the $O(n^5)$ time needed to compute the two polynomials.  Therefore the total deterministic running time is $O(n^5)$.
\end{proof}

\begin{rem}
The preceding argument is not meant to optimize the exponent of the polynomial running time.  Its point is that the fully expanded output has polynomial size: by \(\deg f_w\le n\) there are only \(O(n^3)\) possible monomials, and the coefficient-growth estimates of \cref{sec:coefficient-growth} give only \(O(n)\)-bit coefficients.  Thus the exponential growth of the coefficients as integers does not force exponential-time computation, because exponential magnitude corresponds to only linear bit length.  The elementary implementation above gives the explicit bounds \(O(n^5)\) time and \(O(n^4)\) space on a deterministic Turing machine.  By \cref{cor:random-expanded-output-lower,cor:cubic-output-lower}, the sparse expanded output has size at least \(\Omega(n^3)\) bits both exponentially generically in the positive, freely reduced, and uniformly cyclically reduced models considered above, and in the worst case.  These are output-size, and therefore running-time, lower bounds; they are not meant as lower bounds on auxiliary working space, except for whatever space is used to store or write the output itself.  The dense coefficient-table representation used in the algorithm has worst-case size at most \(O(n^4)\) bits.  Thus the naive method is polynomial and within a factor of at most \(O(n^2)\) of the currently proved sparse-output lower bound; closing this gap is one of the open problems below.
\end{rem}

\section{Open problems}\label{sec:open-problems}

We collect here several natural questions suggested by the preceding results.

\begin{prob}[Fluctuations for cyclically reduced words]\label{prob:typical-cyclically-reduced-degree}
Let $W_n$ be a uniformly random cyclically reduced word of length $n$ in $F(a,b)$.
The exact degree formula and \cref{cor:random-cyclically-reduced-degree} imply
\[
  \frac{\deg f_{W_n}}{n}\to \frac{5}{6}
\]
in probability.  Determine the limiting distribution of
\[
  \frac{\deg f_{W_n}-5n/6}{\sqrt n}.
\]
\end{prob}

\begin{prob}[Fluctuations for positive words]\label{prob:positive-clt}
Let $W_n$ be a random positive word of length $n$, with independent letters and
$\Pr(a)=p$, $\Pr(b)=1-p$.  Prove a central limit theorem for
\[
  \deg f_{W_n}
\]
around its mean-order value
\[
  \bigl(1-p(1-p)\bigr)n.
\]
In particular, in the unbiased case $p=1/2$, prove a central limit theorem around
$3n/4$.
\end{prob}

This should be accessible because, for positive words, \cref{prop:positive-degree}
expresses the degree in terms of the number of cyclic transitions between $a$-runs
and $b$-runs.

The preceding cyclically reduced fluctuation problem replaces an earlier degree-formula problem:
\cref{prop:top-homogeneous-exact-degree} now gives an exact combinatorial formula for the degree and
for the top homogeneous term of every cyclically reduced trace polynomial.

\begin{prob}[Worst-case coefficient growth]\label{prob:worst-coeff-growth}
Let
\[
  H(w)=\norm{f_w}_\infty=\max_{i,j,k}|[x^iy^jz^k]f_w|,
  \qquad
  L(w)=\norm{f_w}_1=\sum_{i,j,k}|[x^iy^jz^k]f_w|,
\]
and set
\[
  C_H=\limsup_{n\to\infty}\max_{||w||_{\mathcal X}=n} H(w)^{1/n},
  \qquad
  C_L=\limsup_{n\to\infty}\max_{||w||_{\mathcal X}=n} L(w)^{1/n}.
\]
Determine $C_H$ and $C_L$.
\end{prob}

The one-generator words $w=a^n$ show that both constants are at least the golden
ratio $\varphi$.  The Cauchy-estimate argument in \cref{sec:coefficient-growth}
gives explicit upper bounds with exponential base less than $2.5$, but these bounds
are unlikely to be sharp.  It is not clear whether mixed $a,b$-words can produce a
larger exponential coefficient-growth rate than one-generator powers.

\begin{prob}[Generic coefficient growth]\label{prob:typical-coeff-growth}
For random freely reduced words $W_n$, determine whether the limits
\[
  \lim_{n\to\infty}\frac{1}{n}\log H(W_n),
  \qquad
  \lim_{n\to\infty}\frac{1}{n}\log L(W_n)
\]
exist in probability or almost surely.  Formulate and solve the analogous problem for
random positive words.
\end{prob}

The examples computed in the discussion preceding these notes suggest exponential
coefficient growth in random freely reduced words, but the typical exponential constants are
not known.

\begin{prob}[Sparsity]\label{prob:sparsity}
Let
\[
  N(w)=\#\{(i,j,k):[x^iy^jz^k]f_w\ne0\}.
\]
Determine the worst-case and generic growth of $N(w)$ for cyclically reduced words of
length $n$.
\end{prob}

The degree bound gives the universal estimate $N(w)=O(n^3)$, while \cref{thm:general-syllable-support-lower,thm:quadratic-support-lower,thm:random-positive-support,thm:random-freely-reduced-support} give quadratic lower bounds for an explicit sequence and, with exponentially high probability, for random positive and random freely reduced words.  It would be interesting to know whether the worst case is actually of order $n^3$, and whether random freely reduced words typically have cubic support.

\begin{prob}[Finer deterministic complexity]\label{prob:finer-complexity}
Determine the optimal deterministic Turing-machine complexity of computing the fully
expanded polynomial $f_w(x,y,z)$ from a word $w$ of length $n$.
\end{prob}

\Cref{thm:polytime-expanded} gives a dense-array algorithm running in time $O(n^5)$ and
space $O(n^4)$ on a deterministic Turing machine.  The expanded dense coefficient table has at most $O(n^4)$ bits with the representation used there, while \cref{cor:random-expanded-output-lower,cor:cubic-output-lower} give $\Omega(n^3)$ lower bounds for sparse binary output, exponentially generically in the positive, freely reduced, and uniformly cyclically reduced models, and in the worst case.  Natural refinements include the following:
\begin{enumerate}[label=(\alph*)]
  \item Determine the true worst-case sparse coefficient bit-size
  \[
    S(n)=\max_{||w||_{\mathcal X}=n}\norm{f_w}_{\bit,1}
  \]
  for fully expanded sparse output, and compare it with the dense-table size.
  \item Decide whether $f_w$ can be computed in time $O(n^4\operatorname{polylog} n)$,
  or more generally in time nearly linear in the output size.
  \item Develop an output-sensitive algorithm whose running time is bounded by a
  small polynomial in $n$ times the number of nonzero monomials $N(w)$.
\end{enumerate}
The present $O(n^5)$ bound is not intended to be optimal; it is an elementary
consequence of the fixed three-variable character-ring description and the linear
bit-size bound for coefficients.

\begin{prob}[Character equivalence, translation equivalence, and complexity]\label{prob:char-translation-complexity}
Study the interaction between the polynomial-time computation of Fricke trace
polynomials and the known equivalence relations on free-group elements. In particular,
compare the complexity of deciding $\SL(2,\C)$-character equivalence,
which by \cref{lem:character-equivalence-fricke,cor:polytime-character-equivalence}
amounts to deciding equality of two Fricke polynomials, with the complexity of deciding
translation equivalence in the sense of Kapovich--Levitt--Schupp--Shpilrain
\cite{KapovichLevittSchuppShpilrain2007}. More generally, determine how the
quantitative invariants studied here, such as degree, support size, and coefficient
height, behave on known families of character-equivalent or translation-equivalent
words, and compare these phenomena with the higher-rank trace-equivalence questions
studied by Lawton--Louder--McReynolds \cite{LawtonLouderMcReynolds2017}.
\end{prob}

\begin{prob}[Higher-rank analogues]\label{prob:higher-rank-analogues}
Develop analogues of the degree, coefficient-growth, and complexity estimates in these
notes for trace polynomials of words in $F_r$ for $r\ge3$ with respect to standard
$\SL(2,\C)$ trace-coordinate systems, and for higher-rank target groups such as
$\SL(3,\C)$.
\end{prob}

For $\SL(2,\C)$ and $F_r$, Horowitz's theorem expresses word traces as polynomials in
finitely many trace coordinates indexed by nonempty ordered products of generators.
For $\SL(3,\C)$, the coordinate rings and trace identities are substantially more
complicated, and the corresponding growth and complexity problems are largely open.
Lawton's work on pairs of $3\times3$ unimodular matrices and on $\SL(3,\C)$ character
varieties of free groups \cite{Lawton2007Pairs,Lawton2008Minimal,Lawton2010AlgInd}
provides explicit coordinate systems in which such questions might be formulated.
The paper of Lawton--Louder--McReynolds \cite{LawtonLouderMcReynolds2017} suggests a
particularly relevant complementary direction: compare the quantitative complexity
of computing trace functions with the complexity of using traces and finite-dimensional
representations to separate conjugacy classes.

\begin{prob}[Higher-rank trace-equivalence complexity]\label{prob:higher-rank-trace-equivalence}
For $n\ge3$, study the expanded trace polynomial of a word $w\in F_2$ as a regular
function on the $\SL(n,\C)$ character variety of $F_2$.  Determine analogues of the degree,
coefficient-growth, support-size, and computation-complexity bounds proved here for
$n=2$.  In particular, relate these quantitative invariants to the problem of whether
non-conjugate words can be $\SL(n,\C)$-trace equivalent, as considered by
Lawton--Louder--McReynolds~\cite{LawtonLouderMcReynolds2017}.
\end{prob}

\section{Disclosure of AI use}

The preparation of this paper substantially relied on chats with ChatGPT, but the author verified all of the proofs given, reviewed and edited the content as needed and takes full responsibility for the content of the paper.

\end{document}